\documentclass[11pt]{article}
\usepackage{graphicx}
\usepackage{color}
\usepackage{amsmath}
\usepackage{amssymb}
\usepackage{amsthm}
\usepackage{amsfonts}
\usepackage{enumerate}
\usepackage{graphicx}
\usepackage{hyperref}
\usepackage[top=2cm,right=2cm, left=2cm, bottom=2cm]{geometry}

\theoremstyle{plain}
\newtheorem{thm}{Theorem}[section]
\newtheorem{lem}[thm]{Lemma}
\newtheorem{prop}[thm]{Proposition}
\newtheorem{cor}[thm]{Corollary}
\newtheorem{remark}[thm]{Remark}
\theoremstyle{definition}

\begin{document}
\title{Long-time asymptotics for nonlocal porous medium equation with absorption or
convection}
\author{Filomena Feo~\thanks{Dipartimento di Ingegneria, Universit\`{a} degli Studi di
Napoli \textquotedblleft Pathenope\textquotedblright, Centro
Direzionale Isola C4 80143 Napoli, Italy. E--mail:
filomena.feo@uniparthenope.it; bruno.volzone@uniparthenope.it }
\and
Yanghong Huang~\thanks{School of Mathematics, The University of Manchester, Oxford Road,
Manchester, M13 9PL, United Kingdom. Email: yanghong.huang@manchester.ac.uk}  	\and
       Bruno Volzone~\footnotemark[1]
       }

\date{\today}

\maketitle

\begin{abstract}
  In this paper, the long-time asymptotic behaviours of nonlocal porous medium equations
 with absorption or convection are studied. In the parameter regimes when the nonlocal diffusion
 is dominant, the entropy method is adapted in this context to derive the exponential convergence
 of relative entropy of solutions in similarity variables.
\end{abstract}

%%%%%%%%%%%%%%%%%%%%%%%%%%%%%%%%%%%%%%%%%%%%%%%%%%%%%%%%%%%%%%%%%%%%%%%%%%
%%%%%%%%%%%%%%%%%%%%%%%%%%%%%%%%%%%%%%%%%%%%%%%%%%%%%%%%%%%%%%%%%%%%%%%%%%
\section{Introduction}
A large variety of models for conserved quantities in continuum mechanics or physics
are described by the continuity equation
$
  u_\tau + \nabla \cdot (u\mathbf{v})=0,
$
where the density distribution $u(y,\tau)$ evolves in time $\tau$
following a velocity field $\mathbf{v}(y,\tau)$. According to Darcy's law,
the velocity $\mathbf{v}$ is usually derived from a potential $p$ in the form $\mathbf{v}
=-\mathcal{D}\nabla p$ for some tensor $\mathcal{D}$. In porous media, the power-law relation $p=u^m$
is commonly proposed, leading to one of the canonical nonlinear diffusion equations~\cite{MR2286292}.
Although local constitutive relations like $p=u^m$ were successful in numerous practical
models, there are situations where the potential (or pressure) $p$ depends non-locally on the
density distribution $u$~\cite{MR1918950}. The simplest prototypical example is
$p=(-\Delta)^{-s}u$, expressed as the Riesz potential of $u$, i.e.,
\begin{equation}\label{eq:riesz}
p(y,t)=(-\Delta)^{-s}u(y,t) = C_{N,s}\int_{\mathbb{R}^N} |y-z|^{2s-N}u(z,t)\mathrm{d}z,
\end{equation}
where the constant $C_{N,s} = \pi^{-N/2}2^{-s}\Gamma(N/2-s)/\Gamma(s)$ is written
in terms of the Euler Gamma function $\Gamma(z)$.
The resulting evolution equation then becomes
\begin{equation}\label{eq:pmep}
  u_\tau - \nabla\cdot(u\nabla(-\Delta)^{-s}u)=0,
\end{equation}
and basic questions like existence, uniqueness and regularity of solutions
have been studied thoroughly in~\cite{vazquez1,vazquez2},
followed by generalisations to other related models~\cite{MR3294409,SDV18}.
While in general it is difficult to obtain quantitative properties of solutions to
non-local nonlinear equations, Eq.~\eqref{eq:pmep} possesses  special features
that enable one to study the long term behaviours
in terms of its self-similar solution. The self-similar profile,
also called \emph{Barenblatt profile}, was initially characterized by an obstacle
problem~\cite{vazquez1}, and was then explicitly constructed  in~\cite{MR3294409}.
Using similarity variables motivated from the scaling relation,
the transformed equation has an entropy function so that
the convergence towards the self-similar profile in one dimension can be established
by the well-known entropy method in~\cite{MR3279352}.

In this paper, we consider two variants of the porous medium equation~\eqref{eq:pmep} with
nonlocal pressure. If the medium is lossy, the continuous density is dissolved,  leading to
the porous medium equation with absorption
\begin{equation}\label{eq:pmea}
  u_\tau - \nabla\cdot(u\nabla(-\Delta)^{-s}u)=-u^r.
\end{equation}
On the other hand, if the density moves with a density dependent local velocity, the equation with
convection becomes
\begin{equation}\label{eq:pmec}
  u_\tau - \nabla\cdot(u\nabla(-\Delta)^{-s}u)=-\mathbf{b}\cdot\nabla u^q,
\end{equation}
for some constant vector $\mathbf{b}\in \mathbb{R}^N$.
In general, when dealing with solutions of both signs, the right-hand side of \eqref{eq:pmea} or
\eqref{eq:pmec} should be replaced by $|u|^{r-1}u$ and $-\mathbf{b}\cdot\nabla
\big(|u|^{q-1}u\big)$
respectively. However, if the initial data are non-negative as we assume below,
the solutions are also non-negative on the time interval of existence,
by following the arguments of \cite{vazquez2,MR3294409}.

For the classical heat equation or porous medium equation,
such variants with absorption~\cite{MR748242,MR868174,MR959221,MR1135955} or convection~\cite{
MR1124296,MR1470816,MR1708659}
have been studied intensively in the past, where the long term asymptotic behaviours
depend on the interplay between the diffusion and absorption/convection,
usually dictated by appropriate scaling transforms.
However,  for Eq.~\eqref{eq:pmea} or~\eqref{eq:pmec} of our interests here, the presence of
nonlinear nonlocal diffusion makes refined techniques such the comparison principle
no longer valid, presenting great challenges in the investigation of quantitative
properties of the resulting solutions.\\ \noindent
In this paper,  we will focus on the long time behaviours of solutions to~\eqref{eq:pmea}
and~\eqref{eq:pmec} mainly in the regime when the nonlocal nonlinear diffusion is dominant,
such that the absorption or convection essentially becomes perturbations added to
Eq.~\eqref{eq:pmep}.
As a result, the entropy method  that is essential in establishing the precise convergence
rate of Eq.~\eqref{eq:pmep} in~\cite{MR3279352} will be adapted in our context.
To proceed, basic estimates for solutions of Eq.~\eqref{eq:pmea} and~\eqref{eq:pmec} will be reviewed
first in Section 2, and the main results about the convergence of solutions will be stated.
In Section 3, the exponential convergence of the relative entropy
between the solutions to Eq.~\eqref{eq:pmea} and their time-dependent Barenblatt profiles
are proved, followed by the proof of similar convergence rates for Eq.~\eqref{eq:pmec} with
convection  in Section 4.
We conclude this paper in Section \ref{generalisation} with generalisations
to related equations and other open problems.

%%%%%%%%%%%%%%%%%%%%%%%%%%%%%%%%%%%%%%%%%%%%%%%%%%%%%%%%%%%%%%%%%%%%%%%%%%
%%%%%%%%%%%%%%%%%%%%%%%%%%%%%%%%%%%%%%%%%%%%%%%%%%%%%%%%%%%%%%%%%%%%%%%%%%
\section{Basic estimates and main results}

Before discussing long-term asymptotic behaviours of Eq.~\eqref{eq:pmea}
or~\eqref{eq:pmec}, we first review a few
fundamental questions about the existence and uniqueness of the solutions in appropriate
spaces. To treat initial data as wide as possible, the usual notion of
weak solutions using test functions can be introduced,
as in~\cite{vazquez2,MR3082241,MR3294409,SFV16} for related equations,
provided that the extra terms corresponding to absorption or convection are  well-defined.
The proof of the existence such weak solutions is now standard, mainly by
taking the limit of regularised equations, by adding linear diffusion, removing
the singularity in the nonlocal operator and confining the solution on bounded domains
(see~\cite{vazquez2,SFV16} for more details).
Here for simplicity, we only focus on the simplest regularisation technique with linear diffusion,
and  define the weak solutions as the limit of solutions to the regularized problems
\begin{equation}\label{eq:pmeareg}
\frac{\partial u^\epsilon}{\partial \tau} -
\nabla\cdot(u^\epsilon\nabla(-\Delta)^{-s}u^\epsilon)=-(u^\epsilon)^r
+ \epsilon \Delta u^\epsilon
\end{equation}
to Eq.~\eqref{eq:pmea} or
\begin{equation}\label{eq:pmecreg}
\frac{\partial u^\epsilon}{\partial \tau}
- \nabla\cdot(u^\epsilon\nabla(-\Delta)^{-s}u^\epsilon)
=-\mathbf{b}\cdot\nabla (u^\epsilon)^q + \epsilon \Delta u^\epsilon,
\end{equation}
to Eq.~\eqref{eq:pmec}, without resorting to more complicated functional
theoretic settings or approximation methods.
The convergence of the regularized
solutions $u^\epsilon$ to the weak solutions of Eq.~\eqref{eq:pmea} or~\eqref{eq:pmec}
is only assumed to be weak* for
non-negative measures, so that the main results are still valid, independent of
the particular type of convergence for the limiting solution sequences in other (more regular) spaces.
However, the question about the uniqueness of these weak solutions
remains open, except in some special cases in one dimension~\cite{ChasJako}. To focus on
the asymptotic  behaviours of our interests here, the weak solutions obtained by
taking limits from~\eqref{eq:pmeareg} or~\eqref{eq:pmecreg} are assumed to be
unique and are the only physically relevant ones below.

Before showing basic estimates of the solutions, we first need several inequalities related
to the nonlocal operator, in standard $L^p(\mathbb{R}^N)$ space with the norm
$\|f\|_p = \left( \int_{\mathbb{R}^N} |f|^p \right)^{1/p}$. The first inequality is related to
the Riesz potential operator $(-\Delta)^{-s}$ defined in~\eqref{eq:riesz},
or equivalently as a multiplier $|\xi|^{-\alpha}$ in the Fourier space, i.e.,
$\mathcal{F}\big[ (-\Delta)^{-\alpha/2} f\big](\xi) =
|\xi|^{-\alpha}\mathcal{F}(\xi)$.
\begin{lem}[Hardy-Littlewood-Sobolev inequality] Let $0<s<N$ and $1<p<q<\infty$ with the relation
	$1/q=1/p-s/N$. If $f \in L^p(\mathbb{R}^N)$, then $(-\Delta)^{-s/2}f
	\in L^q(\mathbb{R}^N)$. Moreover, there is a constant $C(N,s,p)$ such that
	\begin{equation}\label{eq:HLS}
	\|(-\Delta)^{-s/2}f\|_q \leq C(N,s,p)\|f\|_p.
	\end{equation}
\end{lem}

Associated with the Riesz operator is its inverse, the fractional Laplacian $(-\Delta)^{\alpha/2}$,
as a multiplier $|\xi|^{\alpha}$ in the Fourier space, i.e, $\mathcal{F}\big[ (-\Delta)^{\alpha/2} f\big](\xi) =
|\xi|^{\alpha}\mathcal{F}(\xi)$.
The natural space related to the fractional Laplacian is
$\dot{H}^{\alpha/2}(\mathbb{R}^N)$, the set of all functions $f$ such that
\[
\|f\|_{\dot{H}^{\alpha/2}}:= \|(-\Delta)^{\alpha/2}f\|_2 = \left(\int_{\mathbb{R}^N} |\xi|^\alpha |\mathcal{F}[f](\xi)|^2 \mathrm{d}\xi
\right)^{1/2} < \infty.
\]

Although one can not manipulate identities about functions with nonlocal operators in the same ways
as for classical derivatives, there are still powerful inequalities like the following ones
that are crucial in the proof of basic estimates below.
\begin{lem}[Stroock-Varopoulos inequality] For $p\geq1, \alpha \in (0,2]$,
  and any non-negative function $w	\in C_c^\infty(\mathbb{R}^N)$, the following inequality holds true
	\begin{equation}\label{eq:svineq}
	\int w^p (-\Delta)^{\frac{\alpha}{2}} w
	\geq \frac{4p}{(p+1)^2}\int \left|(-\Delta)^{\frac{\alpha}{4}} w^{\frac{p+1}{2}}\right|^2
	=\frac{4p}{(p+1)^2}\int w^{\frac{p+1}{2}} (-\Delta)^{\frac{\alpha}{2}} w^{\frac{p+1}{2}}.
	\end{equation}
\end{lem}
%In the proof of various estimates, the following fractional version of classical Gagliardo-Nirenberg
%type inequalities are heavily used, and are well documented in the survey~\cite{MR2944369}.
\begin{lem}[Fractional Gagliardo-Nirenberg inequality]\label{lem:GNi}
	1) Let $0<\alpha\leq2$. If $u$ is a non-negative function such that $u \in L^1(\mathbb{R}^N)$ and $u^{m/2} \in \dot{H}^{\alpha/2}(\mathbb{R}^N)$
	for some $m>2$, then
	$u \in L^l(\mathbb{R}^N)$ for any $l \in (m/2,m)$ and
	\begin{equation}\label{eq:fGNi}
	\|u\|_{l}^a \leq C_{N,\alpha} \|(-\Delta)^{\alpha/4}u^{m/2}\|_2^2 \|u\|_1^b,
	\end{equation}
	where $C_{N,\alpha}$ is a constant depending only on $N$ and $\alpha$,
	\[
	a = \frac{l(\alpha+N(m-1))}{N(l-1)},\qquad
	b = a-m= \frac{\alpha l + N(m-l)}{N(l-1)}.
	\]
	2) Let $l\geq1$,  $0<\alpha<\min{\{N,2\}}$. %There exists a constant $C$%=C(N,\gamma)\frac{8l}{(l+2)^2}$
	Then for any $v \in L^l(\mathbb{R}^N)\bigcap \dot{H}^{\alpha}(\mathbb{R}^N)$,
	\begin{equation}\label{GN}
	\|v\|_{m}^{\theta+1}\leq C_{1} \|(-\Delta)^{\frac{\alpha}{2}}v\|_2\|v\|_l^\theta
	\end{equation}
	where $m=\frac{N(l+2)}{2(N-\alpha)}$, $\theta=\frac{l}{2}$ and the constant $C_{1}=C(N,\alpha)\frac{(l+2)^{2}}{8(l+1)}$.
\end{lem}
The proof of the Strook-Varopoulos inequality can be found in~\cite{MR1409835}, and
for the proof of \eqref{eq:fGNi} or \eqref{GN}, see~\cite[Lemma 3.2]{MR3294409} and~\cite[Lemma
5.3]{genfrac}, or the survey~\cite{MR2944369}.
With all the preliminary inequalities,  we can derive the basic estimates for solutions
to~\eqref{eq:pmea} or~\eqref{eq:pmec}.
%The proof of these estimates are based on the techniques in~\cite{vazquez2,MR3294409},
%while more

%is also given for completeness.

%\color{red}{
%We observe that a solution to $u^\epsilon$ to Eq.~\eqref{eq:pmeareg} or Eq.~\eqref{eq:pmecreg}
%with non-negative initial data is non-negative. Let us consider Eq.~\eqref{eq:pmeareg} with $|u^\epsilon|^r$ instead of $(u^\epsilon)^r$.  We can argue as in \cite{vazquez2}. Formally at time $t=t_0$ the function $\underset{x\in\mathbb{R}^N}\inf u^\epsilon(x,t_0)=u^\epsilon(x_0,t_0)$ has non-negative time derivative and then $u^\epsilon(x_0,t)$ is non-decreasing in a neighbourhood $[t_0-\delta,t_0+\delta]$ and $\underset{x\in\mathbb{R}^N}\inf u^\epsilon(x,t_0)\geq u^\epsilon(x_0,t_0-\delta)\geq \underset{x\in\mathbb{R}^N}\inf u^\epsilon(x,t_0-\delta)=u^\epsilon(x_1,t_0-\delta)$. Iterating we get $\underset{x\in\mathbb{R}^N}\inf u^\epsilon(x,t_0)\geq u_0(\overline{x})\geq0$ for some $\overline{x}.$

%SCRIVERE PER Eq.~\eqref{eq:pmeareg}/Eq.~\eqref{eq:pmecreg} O PER Eq.~\eqref{eq:pmea} / Eq.~\eqref{eq:pmec}?

%METTERE IL MODULO ALL'INIZIO?

%}\normalcolor

\begin{prop}[Basic estimates for equation with absorption] \label{u:estim}
  Let $u$ be the weak solution to Eq. \eqref{eq:pmea} obtained as the limit
  of $u^\epsilon$ to Eq.~\eqref{eq:pmeareg}
	with non-negative initial data $u^\epsilon(0)=u_0$, $r>1$ and $1\leq p \leq \infty$. If $u_0\in L^p(\mathbb{R}^N)\cap L^1(\mathbb{R}^N)$, we have
	\begin{equation}\label{estim:lp}
	\|u(\tau)\|_p\leq  \|u_0\|_p, \quad \quad \tau\geq0.
	\end{equation}
	Moreover if $u_0\in L^1(\mathbb{R}^N)$ the following estimate
	\begin{equation}\label{decay:lp}
	\|u(\tau)\|_p\leq C(\|u_0\|_1,N,p,r,s)
	\tau^{-\max\left(\frac{N(p-1)}{p(N+2-2s)},\frac{p-1}{p(r-1)}\right)},
	\quad \quad \tau>0,
	\end{equation}
	holds for a positive constant $C(\|u_0\|_1,N,p,r,s)$ depending on $N,p,r,s$ and on the initial condition only through the total mass $\|u_0\|_1$.
\end{prop}

\begin{proof} The non-negativity of the solution $u^\epsilon(\tau)$
  can be proved similarly as in~\cite{vazquez2,MR3294409}, and is skipped for simplicity.
Then the fact $\frac{d}{d\tau}\|u^\epsilon(\tau)\|_1 = -\int \big(u^\epsilon(\tau)\big)^r \leq 0$
immediately implies that $\|u^\epsilon(\tau)\|_1 \leq \|u^\epsilon(0)\|_1=\|u_0\|_1$.
Similarly, for  any finite $p>1$,
\begin{align}\label{eq:dudtau}
  \frac{1}{p}\frac{d}{d\tau} \|u^\epsilon(\tau)\|^p_p
  &= \epsilon\int (u^\epsilon)^{p-1} \Delta u^\epsilon
  +\int \nabla\cdot\big(u^\epsilon\nabla (-\triangle)^{-s}u^\epsilon\big)  (u^\epsilon)^{p-1}
  -\int  (u^\epsilon)^{r+p-1} \cr
  &= -(p-1)\epsilon\int (u^\epsilon)^{p-2}|\nabla u^\epsilon|^2
  - \frac{p-1}{p}\int (u^\epsilon)^p (-\Delta)^{1-s}u^\epsilon - \int (u^\epsilon)^{r+p-1} \cr
  &\leq
  - \frac{4(p-1)}{(p+1)^2}\int \left|(-\Delta)^{\frac{1-s}{2}}(u^\epsilon)^{\frac{p+1}{2}}\right|^2
  - \int (u^\epsilon)^{r+p-1},
\end{align}
where the Stroock-Varopoulos inequality~\eqref{eq:svineq}
is used in the last step. Since all terms in the last line of~\eqref{eq:dudtau} are non-positive,
we can conclude that $\|u^\epsilon(\tau)\|_p \leq \|u^\epsilon(0)\|_p$,
which is also true for $p=\infty$.

To obtain refined decay rate along the evolution,  the last two terms in~\eqref{eq:dudtau}
can be related to $\|u_\epsilon(\tau)\|_p$ to establish a self-contained differential inequality.
By choosing $l=p, \alpha=2-2s$ and $m=p+1$, the fractional Gagliardo-Nirenberg
inequality~\eqref{eq:fGNi} becomes
\begin{equation*}\label{E1}
  \big\|u^\epsilon\big\|_p^a
  \leq C_{N,\alpha} \left\|(-\Delta)^{\frac{1-s}{2}}\big(u^\epsilon\big)^{\frac{p+1}{2}}
  \right\|_2^2\big\|u^\epsilon\big\|_1^b
  \leq C_{N,\alpha} \left\|(-\Delta)^{\frac{1-s}{2}}\big(u^\epsilon\big)^{\frac{p+1}{2}}
  \right\|_2^2\big\|u_0\big\|_1^b
\end{equation*}
with $a=\frac{p(Np+2-2s)}{N(p-1)}$ and $b= \frac{N+2p-2sp}{N(p-1)}$.  On the other hand,
we have the H\"{o}lder inequality
\begin{equation*}
  \big\|u^\epsilon\big\|_p \leq
  \big\|u^\epsilon\big\|_{r+p-1}^{\frac{(r+p-1)(p-1)}{(r+p-2)p}}
  \big\|u^\epsilon\big\|_1^{\frac{(r-1)}{p(r+p-2)}}
  \leq  \big\|u^\epsilon\big\|_{r+p-1}^{\frac{(r+p-1)(p-1)}{(r+p-2)p}}
  \big\|u_0\big\|_1^{\frac{(r-1)}{p(r+p-2)}}.
\end{equation*}
Therefore, the differential inequality~\eqref{eq:dudtau}
suggests the following two  decay rates of $\big\|u^\epsilon(\tau)\big\|_p$,
\[
  \frac{1}{p} \frac{d}{d\tau} \|u^\epsilon(\tau)\|_p^p
  \leq  -\frac{4(p-1)}{C_{N,\alpha}(p+1)^2\|u_0\|_1^b} \|u^\epsilon(\tau)\|_p^a
  \qquad \mbox{and}\qquad
  \frac{1}{p} \frac{d}{d\tau} \|u^\epsilon(\tau)\|_p^p
  \leq - \frac{\|u^\epsilon\|_p^{\frac{p(r+p-2)}{p-1}}}{\|u_0\|_1^{\frac{r-1}{p-1}}},
\]
from which we get
\[
 \|u^\epsilon(\tau)\|_p  \leq
\left(
 \frac{4p(N+2-2s)}{NC_{N,\alpha}(p+1)^2\|u_0\|_1^b}\tau
  \right)^{-\frac{N(p-1)}{p(N+2-2s)}}
  \quad \mbox{and}\quad
  \|u^\epsilon(\tau)\|_p \leq
  \left(
  \frac{p(r-1)}{p-1}\|u_0\|_1^{-\frac{r-1}{p-1}}\tau
  \right)^{-\frac{p-1}{p(r-1)}}.
\]
These two bounds lead to the estimate
$\|u^\epsilon(\tau) \|_p \leq C(\|u_0\|_1,N,p,r,s)\tau^{-\max\left(\frac{N(p-1)}{p(N+2-2s)},
\frac{p-1}{p(r-1)}\right)}$, where the constant depends only on certain powers of $\|u_0\|_1$.
The desired estimates~\eqref{estim:lp} and~\eqref{decay:lp} then hold by taking the limit as $\epsilon$ goes to zero.

However, the case for $p=\infty$ in~\eqref{decay:lp} can
not be obtained directly by taking the limit  as $p$ goes to infinity, because the constant
$C(\|u_0\|_1,N,p,r,s)$ becomes infinity. Instead, the proof in~\cite[Theorem 6.1]{MR3294409} for
$L^1-L^\infty$ smoothing effect of solutions  is adapted.
\end{proof}

\begin{remark}
The two decay rates in~\eqref{decay:lp} actually correspond to two different regimes
for the behaviours of the solutions: in the diffusion dominated regimes
with $r>(2N+2-2s)/N$, the decay rate $\tau^{-N(p-1)/p(N+2-2s)}$ prevails and the constant $C$ does not depend on $r$; otherwise
the diffusion and the absorption balance each other, and the rate $\tau^{-(p-1)/p(r-1)}$
prevails. We will mainly focus  on the first case, where the entropy method can be applied
in the transformed equation with similarity variables.
\end{remark}

For Eq.~\eqref{eq:pmec} with convection, similar computation can be applied
to obtain the following estimates.

\begin{prop}[Basic estimates for equation with convection] \label{u:estim3}
	Let $u$ be the weak solution to Eq.~\eqref{eq:pmec} obtained as the limit of $u^\epsilon$ to
    Eq.~\eqref{eq:pmecreg} with non-negative initial data $u^\epsilon(0)=u_0$, $q>1$ and
 $1\leq p \leq \infty$. If $u_0 \in L^1(\mathbb{R}^N) \cap L^p(\mathbb{R}^N)$,  we have
	\begin{equation*}%\label{estim:lp3}
	\|u(\tau)\|_p\leq  \|u_0\|_p, \quad \quad \tau\geq0.
	\end{equation*}
	Moreover if $u_0 \in L^1(\mathbb{R}^N)$ the following inequality
	\begin{equation*}%\label{decay:lpc}
	\|u(\tau)\|_p\leq C(\|u_0\|_1,N,p,q,s)
	\tau^{-\frac{N(p-1)}{p(N+2-2s)}},
	\quad \quad \tau>0,
	\end{equation*}
	holds for a positive constant $C(\|u_0\|_1,N,p,q,s)$ depending on $N,s,p$ and on the initial total mass $\|u_0\|_1$.
\end{prop}

In this paper, we concentrate on cases where the absorption in Eq.~\eqref{eq:pmea}
or the convection in Eq.~\eqref{eq:pmec} is dominated by the nonlocal diffusion, so that
the long term asymptotic behaviours are essentially governed by Eq.~\eqref{eq:pmep}.
As a result,  the same change of similarity variables used to study Eq.~\eqref{eq:pmep}
in~\cite{vazquez1,MR3294409,MR3279352} is adopted here, that is,
\begin{equation} \label{eq:simivar}
x=y(1+\lambda\tau)^{-1/\lambda}, \quad t=\frac{1}{\lambda}\log(1+\lambda\tau),\quad
\rho(x,t) = (1+\lambda\tau)^{N/\lambda}u(y,\tau)
\end{equation}
with $\lambda =N+2-2s$. In this way, the absorption or convection becomes exponentially small
perturbations as shown in the next two sections. To better illustrate the main techniques involved,
we first review the convergence of the transformed equation
\begin{equation}\label{eq:smfracpme}
\rho_t - \nabla_x \cdot(\rho\nabla_x (-\Delta)^{-s} \rho+x\rho)=0,
\end{equation}
obtained from Eq.~\eqref{eq:pmep} via the similar variables~\eqref{eq:simivar}.
It is show in~\cite{vazquez1} that, the solution $\rho(x,t)$ converges
to a steady state, which can be characterized as the solution of a fractional obstacle problem.
The steady state $\rho_M(x)$, depending on the total conserved mass $M =
\int \rho(x,0)\mathrm{d}x= \int u(x,0)\mathrm{d}x$, is given explicitly in~\cite{MR3294409} as
\begin{equation}\label{eq:profile}
\rho_{M}(x)
= \frac{2^{2s-1}\Gamma(1+N/2)}{\Gamma(2-s)
	\Gamma(1-s+N/s)}\Big( R^2 - |x|^2\Big)_{+}^{1-s},
\end{equation}
where the radius of support $R$ is determined by the total conserved mass $M$ through the relation
\begin{equation}
  M = \int   \rho_{M}(x) \mathrm{d}x
= \frac{2^{2s}\pi^{N/2}\Gamma(1+N/2)}{(N+2-2s)\Gamma(1-s+N/2)^2}R^{N+2-2s}.\label{mass}
\end{equation}
Because of the presence both nonlinearity and nonlocality, the convergence of the solution
$\rho(x,t)$ towards $\rho_M(x)$ is more difficult,
relying heavily on the Lyapunov function or the entropy
\begin{equation}\label{eq:entropy}
H[\rho] = \int \frac{1}{2}\rho (-\Delta)^{-s}\rho
+ \frac{1}{2}|x|^2\rho.
\end{equation}
Since $H[\rho]$ is a convex functional, it is proved in~\cite{MR3262506} that
the steady state~\eqref{eq:profile} is  the unique minimizer of $H[\rho]$ on the space
of all non-negative measures  with total mass $M$ and is characterized by the relation
\begin{equation}\label{eq:profrel}
(-\Delta)^{-s}\rho_M(x) + \frac{1}{2}|x|^2 =
\frac{N}{2(N-2s)}R^2,\qquad |x| \leq R,
\end{equation}
and
\begin{equation}\label{eq:profre2}
(-\Delta)^{-s}\rho_M(x) + \frac{1}{2}|x|^2 \geq
\frac{N}{2(N-2s)}R^2,\qquad |x| > R.
\end{equation}
However, more refined questions like the convergence rate of $\rho(x,t)$ towards $\rho_M(x)$
seem less likely to be answered in the framework of this classical notion of convexity.
Instead, it is more convenient to study the \emph{displacement convexity}
originated by McCann~\cite{MR1451422}, a key concept in the theory of optimal
transport~\cite{MR1964483,MR3409718}: by defining the entropy dissipation rate
\begin{equation} \label{eq:disp}
  I[\rho]=\int \rho \left|\nabla (-\Delta)^{-s}\rho + x\right|^2,
\end{equation}
which is exactly $-\frac{d}{dt}H[\rho]$ when $\rho$ is governed by~\eqref{eq:smfracpme},
the main task is to establish a relationship between the relative entropy $H[\rho|\rho_M]:=H[\rho]-H[\rho_M]$ and the
entropy dissipation rate $I[\rho]$. Such a relationship, also called \emph{entropy-entropy
dissipation inequality}, lies at the heart of the entropy method in proving convergence
and is established in  only one dimension in~\cite{MR3279352} as follows.
\begin{lem} Let $N=1$, $s<1/2$. If the entropy $H[\rho]$ and entropy dissipation rate $I[\rho]$
  are defined as in~\eqref{eq:entropy} and~\eqref{eq:disp} respectively, then  the inequality
\begin{equation}\label{eq:entrdisp}
  H[\rho|\rho_{M}] := H[\rho]-H[\rho_{M}]\leq \frac{1}{2}I[\rho]
\end{equation}
holds for any non-negative Radon measure $\rho$ with the same total mass as $\rho_M$
such that $H[\rho]$ is bounded.
\end{lem}
Formally, once a relationship like~\eqref{eq:entrdisp}
is established, the Gronwall type inequality
\begin{equation}\label{eq:dhdt}
  \frac{d}{dt}H[\rho|\rho_M] = -I[\rho] \leq -2H[\rho|\rho_M]
\end{equation}
implies the exponential convergence of the relative entropy, i.e.,
$H[\rho(t)|\rho_M] \leq e^{-2t}(H[u_0|\rho_M])$.
Exponential convergence in other norms or metrics,
for instance using Csisz\'{a}r-Kullback inequality
for norms of $\rho-\rho_M$ dominated by the relative entropy $H[\rho|\rho_M]$ are available in
literature (see the monograph~\cite{MR3497125} for a detailed review). One such inequality
in the present setting is also established in~\cite[Lemma 3.3]{MR3279352}, from which one can show the
exponential convergence of the solution.
\begin{lem} Let $\rho$ be a non-negative function on $\mathbb{R}^N$ with the same
  total mass $M$ as the steady state $\rho_M$, then
	\begin{equation}\label{frac_lap_rel_entropy}
	\|(-\Delta)^{-s/2}(\rho-\rho_M)\|_{2}^2 \leq 2(H[\rho|\rho_{M}]).
	\end{equation}
	\label{lem:hsnorm}
\end{lem}

After all the basic estimates and background information about entropy methods, we now
state the main theorems. They establish the asymptotic behaviours on the one-dimensional case, in terms of the
relative entropy for the transformed variables, and $H^{-s}(\mathbb{R})$ norm of
the difference between the solution $u$ and the asymptotic profile
\begin{equation} \label{eq:uprof}
  u_{M}(y,\tau)
= (1+\lambda\tau)^{-N/\lambda}\rho_M\big(y(1+\lambda\tau\big)^{-1/\lambda}).
\end{equation}
It is clear that \eqref{eq:uprof} is the time-dependant Barenblatt profile for the fractional-pressure evolution equation \eqref{eq:pmep}.
\begin{thm}[Asymptotic convergence for the equation with absorption]\label{thm:abs}
 Let $N=1$, $u(y,\tau)$ be the weak solution of~\eqref{eq:pmea} obtained from the limiting
  solutions of~\eqref{eq:pmeareg}
	with non-negative initial data $u_0\in L^{1}(\mathbb{R},(1+|x|^{2})\mathrm{d}x)\cap
    L^{\infty}(\mathbb{R})$, $r>4-2s$ and $s<1/2$.
    Let $u_{M(\tau)}$ be the rescaled profile~\eqref{eq:uprof} with  the mass $M(\tau)
    =\int u(y,\tau)d y$.
    Then there exists a constant $C$ depending on $r,s,\|u_0\|_1,\|u_0\|_\infty$ and on $H[u_0|u_{M_0}]$ such that
    \begin{equation}\label{eq:relenabs}
      H[u(\tau)|u_{M(\tau)}]
      \leq C(1+|\log\tau|)^2(1+\lambda\tau)^{-\frac{1-2s}{\lambda}-2\min(\frac{1}{\lambda},\delta)}
    \end{equation}
with $\delta=(r-1)/\lambda-1>0$.  In particular, we have
\[
\big\|(-\Delta)^{-s/2} \big(u(\tau)-u_{M(\tau)}\big)\big\|_2
\leq C (1+|\log \tau|)(1+\lambda\tau)^{-\frac{1-2s}{2\lambda}-\min(\frac{1}{\lambda},\delta)}.
\]
\end{thm}

\begin{thm}[Asymptotic convergence for the equation with convection] Let $N=1$, $u(y,\tau)$ be the weak solution of~\eqref{eq:pmec} obtained from the limiting
  solutions of~\eqref{eq:pmecreg}
  with non-negative initial data $u_0\in L^{1}(\mathbb{R},(1+|x|^{2})\mathrm{d}x)\cap
    L^{\infty}(\mathbb{R})$, $r>3-2s$ and $s<1/2$.
    Let $u_{M_0}$ be the rescaled profile~\eqref{eq:uprof} with the conserved mass $M_0=\|u_0\|_1$.
    Then there exists a constant $C$ depending on $q,s,M_0,\|u_0\|_\infty$ and on $H[u_0|u_{M_0}]$ such that
    \begin{equation}\label{eq:relencon}
      H[u(\tau)|u_{M_0}]
      \leq C(1+|\log\tau|)^2(1+\lambda\tau)^{-\frac{1-2s}{\lambda}-2\min(\frac{1}{\lambda},\theta)}
    \end{equation}
with $\theta=q/\lambda-1>0$. In particular we have
    \[
      \big\|(-\Delta)^{-s/2} \big(u(\tau)-u_{M_0}\big)\big\|_2
      \leq C (1+|\log \tau|)(1+\lambda\tau)^{-\frac{1-2s}{2\lambda}-\min(\frac{1}{\lambda},\theta)}.
    \]
    \label{thm:conv}
\end{thm}

The convergence of the solution $u(\tau)$ to $u_{M(\tau)}$ (or $u_M$) in other metrics usually relies on
interpolation inequalities between norms, and requires higher regularity on $u(\tau)-u_{M(\tau)}$
(or $u(\tau)-u_M$) that is out of scope of this paper. This issue will be commented near the end of Subsection 3.4.

%%%%%%%%%%%%%%%%%%%%%%%%%%%%%%%%%%%%%%%%%%%%%%%%%%%%%%%%%%%%%%%%%%%%%%%%%
\section{Fractional diffusion with absorption}

In this section, the long term asymptotic behaviours of non-negative solutions to the
fractional diffusion equation~\eqref{eq:pmea} with absorption are studied.
Different behaviours will be classified first based on the parameter $r$. The focus will on
the regime when  the absorption is dominated by the diffusion and can be ignored in the
long run. Formal computations will be performed to illustrate the entropy
method for smooth enough solutions and then
for more general weak solutions by limiting procedure.

%%%%%%%%%%%%%%%%%%%%%%%%%%%%%%%%%%%%%%%%%%%%%%%%%%%%%%%%%%%%%%%%%%%%%%%%%%%
\subsection{Basic behaviours and the transformed equation in similarity variables}
Because the evolution equation~\eqref{eq:pmea} is governed by two forces,
the fractional diffusion $\nabla\cdot(u\nabla (-\Delta)^{-s}u)$ and
the absorption $-u^r$,  the corresponding long term behaviours of
the solutions are determined by  the competition between these two effects,
similar to the well studied cases where the nonlocal diffusion is replaced with classical
linear  or nonlinear diffusion~\cite{MR748242,MR865069}.
Formally if the non-negative solution $u$ is smooth  and uniformly bounded,
 $\nabla\cdot(u\nabla (-\Delta)^{-s}u)$ is non-positive at any global maximum of $u$ and
 $\|u(\tau)\|_\infty$ is dominated by the
ODE   $u_\tau=-u^r$. In other words,  if $r \in (0,1)$, the solution
vanishes in finite time, and if $r>1$, then
\[
  \|u(\tau)\|_{\infty} \leq
  \Big( \|u_0\|_{\infty}^{1-r} + (r-1)\tau\Big)^{-1/(r-1)}.
\]
Furthermore, if $r\geq1$, the solution can not vanish in finite time.
In fact, from the condition $\|u(t)\|_\infty \leq \|u_0\|_\infty$ in Proposition~\ref{u:estim}, the
change of the total mass of $u(\tau)$ satisfies the differential inequality
\[
  \frac{d}{d\tau} \left(\int u(y,\tau)\mathrm{d}y \right)
  = -\int u(y,\tau)^r\mathrm{d}y
  \geq - \left(\int u(y,\tau)\mathrm{d}y\right) \|u_0\|_\infty^{r-1}.
\]
The lower bound on the decreasing rate implies that the total mass can not
be zero at any finite time.

The regime $r>1$ can be analysed further, depending on the role
played by the nonlocal diffusion dictated by scaling argument. If $r \in \big(1,(2N+2-2s)/N\big)$,
the long term behaviours are the determined by both the fractional
diffusion and the absorption, and the solutions are expected to
converge to the self-similar solution of the form
\begin{equation}\label{eq:Ursmall}
  \Big(1+(r-1)\tau\Big)^{-1/(r-1)}
  U\left(y\big(1+(r-1)\tau\big)^{-\alpha/(r-1)}\right),
\end{equation}
where $\alpha = \frac{2-r}{2(1-s)}$ and
the self-similar profile $U$ satisfies the equation
\[
  U+\alpha x\cdot \nabla U+ \nabla\cdot(U\nabla (-\Delta)^{-s}U) - U^r=0.
\]
When $s=0$ (the case with classical diffusion), the convergence of the solutions towards the
self-similar solutions can be established rigorously~\cite{MR748242,MR865069},
mainly using comparison principles. However, similar refined quantitative techniques
do not seem to be available in the fractional setting;  properties of the self-similar profile $U$
governed by~\eqref{eq:Ursmall} and any further information about the
convergence remain  challenging open problems.

In this paper we are interested in the parameter regime when
\begin{equation}
r > \frac{2N+2-2s}{N}\label{diffdomabs},
\end{equation}
such that the absorption becomes small perturbation in determining the long term behaviours.
More precisely, using the same change of similarity variables~\eqref{eq:simivar},
Eq.~\eqref{eq:pmea} becomes
\begin{equation}\label{eq:smabs}
    \rho_t = \nabla\cdot\left[ \rho\left(\nabla (-\Delta)^{-s}\rho+x\right)
    \right] - P(t)^{-\delta}\rho^r,
\end{equation}
where $P(t)=1+\lambda \tau = e^{\lambda t}$  and $\delta=N(r-1)/\lambda-1>0.$ Similarly,
basic estimates for $u(\tau)$ in Proposition~\ref{u:estim} can be translated directly
into those for $\rho(t)$.

\begin{prop}\label{ro:estim}
  Let the solution $\rho$ for Eq.~\eqref{eq:smabs} be obtained from the solution $u$
  for Eq.~\eqref{eq:pmea} using the change of similarity variables~\eqref{eq:simivar}
 and $1\leq p \leq \infty$. Assume we are in diffusion-dominated regime \eqref{diffdomabs}. If $u_0\in L^1(\mathbb{R}^N)\cap  L^p(\mathbb{R}^N)$, we have
 \begin{equation}\label{estim::lp}
 \|\rho(t)\|_p\leq \|u_0\|_p e^{N\left(1-\frac{1}{p}\right)t}, \quad \quad t\geq0.
 \end{equation}
Moreover if $u_0\in L^1(\mathbb{R}^N)$, the following estimates
  \begin{equation}\label{bound:lp}
 \|\rho(t)\|_p\leq C(\|u_0\|_1,N,p,s), \quad \text{ for } t>t_0
\end{equation}
 \begin{equation}\label{bound::lp}
\|\rho(t)\|_p\leq C(\|u_0\|_1,N,p,s){t}^{-\frac{N(p-1)}{p(N+2-2s)}}, \quad \text{ for } 0<t<t_0
 \end{equation}
hold for some $t_0>0$.
\end{prop}

\begin{proof}
The estimate \eqref{estim::lp} is exactly \eqref{estim:lp},  in terms of the similarity
variables~\eqref{eq:simivar}. In the parameter regime $r>(2N+2-2s)/N$,
the estimate \eqref{decay:lp} becomes
$
\|u(\tau)\|_p \leq C(\|u_0\|_1,N,p,s)\tau^{-N(p-1)/p(N+2-2s)}
$ (with no dependence on $r$ in the constant $C$), or
\[
  \|\rho(t)\|_p \leq  \|u(\tau)\|_p e^{N(1-\frac{1}{p})t}\leq
  C(\|u_0\|_1,N,p,s)
  \left( \frac{1-e^{-\lambda t}}{\lambda}\right)^{-\frac{N}{N+2-2s}\left(1-\frac{1}{p}\right)}.
\]
Therefore the estimate \eqref{bound:lp} is obtained for $t>t_0$ and \eqref{bound::lp} for $0<t<t_0$,
using the bounds $(1-e^{-\lambda t})/\lambda\geq (1-e^{-\lambda t_0})/\lambda$ and
$(1-e^{-\lambda t})/\lambda \geq t(1-e^{-\lambda_0 t})/(\lambda t_0)$ respectively.
\end{proof}

\begin{remark}
Using \eqref{estim::lp} and \eqref{bound:lp}, if $u_0\in L^1(\mathbb{R}^N)\cap L^p(\mathbb{R}^N)$,
then $\|\rho(t)\|_p $ is uniformly bounded for any time $t>0$ by a constant depending on $N$, $s$, $p$, $\|u_0\|_1$
and $\|u_0\|_p$.  In fact, from~\eqref{bound:lp}, the uniform bound of $\|u_0\|_p$ on the initial
condition $u_0$ is not needed for the long term behaviours of our interest.
For simplicity, we will assume that $u_0 \in L^1(\mathbb{R}^N)\bigcap
L^\infty(\mathbb{R}^N)$ in the rest of the paper, such that
\begin{equation}\label{Linfty_uniform_bound}
\|\rho(t)\|_\infty \leq C(\|u_0\|_1,\|u_0\|_\infty,N,s).
\end{equation}
Consequently, by interpolation between norms,
we also have $\|\rho(t)\|_p \leq C(\|u_0\|_1,\|u_0\|_\infty, N,p,s)$ for any
$1\leq p \leq \infty$.
\end{remark}

Let $M(t) = \int \rho(x,t)\mathrm{d}x$ be the total mass at time $t$.
Since $$\frac{d}{dt}M(t) = -P(t)^{-\delta}\int \rho^r\leq 0,$$
the mass $M(t)$ is decreasing, \textit{i.e.} $\|\rho(t)\|_1:=M(t)\leq M_0:=\|u_0\|_1$.
Moreover, we can show that the limiting mass
\begin{equation}\label{M infty}
  M_\infty=\|u_0\|_{1}-\int_0^{+\infty}P(s)^{-\delta} \int \rho(x,s)^r \mathrm{d}x \mathrm{d}s
\end{equation}
is strictly positive. Indeed, from the uniform bound~\eqref{Linfty_uniform_bound} of the solution,
we get
\begin{equation}
    \frac{d}{dt}M(t) = -P(t)^{-\delta}\int \rho^r
    \geq -P(t)^{-\delta} \|\rho(t)\|_\infty^{r-1} \int \rho
    \geq -C(\|u_0\|_1,\|u_0\|_\infty,N,s)P(t)^{-\delta}M(t).
    \label{massevol}
\end{equation}
This differential inequality can be integrated on the time interval $(0,t)$ to obtain
\[
  \log\left( \frac{M(t)}{M(0)}\right)\geq
   - C\big(\|u_0\|_1,\|u_0\|_\infty,N,s\big)\int_{0}^{t}
  P(t)^{-\delta}\mathrm{d}t
  \geq
   - C\big(\|u_0\|_1,\|u_0\|_\infty,N,s\big)\int_{0}^{\infty}  P(t)^{-\delta}\mathrm{d}t.
\]
This implies that
\[
 M(t) \geq M_\infty \geq M(0)\exp\left(
   - C\big(\|u_0\|_1,\|u_0\|_\infty,N,s\big)\int_{0}^{\infty}  P(t)^{-\delta}\mathrm{d}t
  \right)>0.
\]
With this result, we can show that $\rho(x,t)$ eventually converges to the profile $\rho_{M_\infty}(x)$, by first showing that it converges to the time dependent profile $\rho_{M(t)}(x)$.

%%%%%%%%%%%%%%%%%%%%%%%%%%%%%%%%%%%%%%%%%%%%%%%%%%%%%%%%%%%%%%%%%%%%%%%%%%%
%%%%%%%%%%%%%%%%%%%%%%%%%%%%%%%%%%%%%%%%%%%%%%%%%%%%%%%%%%%%%%%%%%%%%%%%%%%
\subsection{Exponential convergence for smooth solutions}
In this subsection, we show that if $\rho(t)$ is a solution to Eq.~\eqref{eq:smabs}
with certain regularity, then it is ``close'' to  the Barenblatt profile $\rho_{M(t)}$
measured in some metric.
Here the main technique is to adapt the well-established entropy method in this new setting,
by proving the exponential convergence of the relative entropy between
the solution $\rho(t)$ and the Barenblatt profile $\rho_{M(t)}$.
The calculation below is entirely formal,
without worrying about the regularity of the solution $\rho(t)$, so that the main idea is conveyed
in the simplest setting. Once this convergence is  established, the convergence for general weak
solutions through limiting procedure will be proved
along the same limiting sequences, as shown in the next subsection.

The key idea is still to establish a self-contained differential equality
for the relative entropy $H[\rho(t)|\rho_{M(t)}]$.
From the definition of the entropy $H[\rho]$,
if $\rho(t)$ is a solution of~\eqref{eq:smabs} and decays to zero fast enough, then
the time change of the entropy becomes
\begin{equation} \label{entropy-dissipabsorp}
\frac{d}{dt}H[\rho(t)] =
-I[\rho(t)]-P(t)^{-\delta}\int \rho(t)^{r}\left[
(-\Delta)^{-s}\rho+\frac{x^2}{2}\right].
\end{equation}
On the other hand, since the time-dependent profile $\rho_{M(t)}$ is supported only on the ball
$|x|\leq R(t)$ with $R(t)$ related to the total mass $M(t)$ via~\eqref{mass},
\begin{align*}
\frac{d}{dt} H[\rho_{M(t)}]
&= \frac{d}{dt} \int_{|x|\leq R(t)} \left(
\frac{1}{2}\rho_{M(t)}(-\Delta)^{-s}\rho_{M(t)}+\frac{1}{2}|x|^2\rho_{M(t)}\right) \cr
&= \int_{|x|\leq R(t)}
\left( (-\Delta)^{-s} \rho_{M(t)}+\frac{|x|^2}{2}\right) \frac{d}{dt} \rho_{M(t)}+
R'(t)\int_{|x|=R(t)}
\frac{\rho_{M(t)}}{2}\Big((-\Delta)^{-s}\rho_{M(t)}+|x|^2\Big) \cr
&=\int_{|x|\leq R(t)}
\left( (-\Delta)^{-s} \rho_{M(t)}+\frac{|x|^2}{2}\right) \frac{d}{dt} \rho_{M(t)},
\end{align*}
where the fact that $\rho_{M(t)}$ vanishes on the boundary $|x|=R(t)$ is used in the last step.
From the characterization~\eqref{eq:profrel} of the local profile $\rho_{M(t)}$,
\[
\frac{d}{dt} H[\rho_{M(t)}]
= \frac{NR(t)^2}{2(N-2s)}\int \frac{d}{dt}\rho_{M(t)}
= \frac{NR(t)^2}{2(N-2s)}M'(t)
= -\frac{NR(t)^2}{2(N-2s)}P(t)^{-\delta}\int \rho^r.
\]
Using the characterization~\eqref{eq:profrel} and~\eqref{eq:profre2} again for
the behaviour of $(-\Delta)^{-s}\rho_{M(t)}+|x|^2/2$ inside and outside the ball
$|x|\leq R(t)$, we finally get
\begin{equation}\label{entropyasympprof}
\frac{d}{dt}H[\rho_{M(t)}] \geq -P(t)^{-\delta}
\int \left( (-\Delta)^{-s}\rho_{M(t)} + \frac{|x|^2}{2} \right)
\rho^r.
\end{equation}
With these preliminary formal computations, now we can show
the exponential convergence of the relative entropy
$H[\rho(t)|\rho_{M(t)}]$ as summarized in the following theorem.
The result above holds only for $N=1$ since
the essential entropy-entropy dissipation inequality \eqref{eq:entrdisp} has only been proved in one dimension. As a result, the diffusion dominated regime \eqref{diffdomabs}
becomes $r>4-2s$ or equivalently $\delta = (r-1)/\lambda-1>0$.

\begin{thm}\label{expdecaysmooth} Let $N=1$ and $\rho$ be a smooth solution of the one-dimensional fractional porous medium equation
	with  absorption~\eqref{eq:smabs} with non-negative initial data $u_0$ and $s<1/2$.
	Let   $\rho_{M(t)}$ be the Barenblatt profile with the same mass $M(t)$ as $\rho(t)$. Then
	we have that the relative entropy $H[\rho|\rho_{M(t)}]:=H[\rho]-H[\rho_{M(t)}]$ decays to zero exponentially fast. That
	is, there is a constant $C$ depending on the mass $M_{0}:=\|u_0\|_1$, the sup norm $\|u_0\|_\infty$, the exponents $r$, $s$ and on $H[u_0|\rho_{M_0}]$, such that
	\begin{equation}
	H[\rho(t)|\rho_{M(t)}] \leq C(1+t)^2 \exp(-2\min(1,\lambda\delta)t).\label{exponentialconv}
	\end{equation}
\end{thm}

%    \right)^{(N+4-2s)/(N+2-2s)}.
%\end{align}
%As a result,
%\begin{align*}
%\frac{d}{dt}H[\rho_{M(t}]&=
%\frac{N}{2(N-2s)}
%\left( \frac{(N+2-2s)\Gamma(1-s+N/2)^2}{2^{2s}\pi^{N/2}\Gamma(1+N/2)}M(t)
%\right)^{2/(N+2-2s)}\frac{d}{dt}M(t)\cr
%&=
%-\exp(-\lambda \delta t)
%\frac{N}{2(N-2s)}
%\left( \frac{(N+2-2s)\Gamma(1-s+N/2)^2}{2^{2s}\pi^{N/2}\Gamma(1+N/2)}M(t)
%\right)^{2/(N+2-2s)}\int \rho^r dx.
%\end{align*}

\begin{proof} The proof is based on a self-contained differential inequality
	of the relative entropy $H[\rho(t)|\rho_{M(t)}]$.
	From the two expressions~\eqref{entropy-dissipabsorp} and~\eqref{entropyasympprof}
	for the time derivative of the entropies,
	we have
	\begin{align}
	\frac{d}{dt}H[\rho(t)|\rho_{M(t)}]
	&\leq - I[\rho(t)] -P(t)^{-\delta}\int \rho(t)^r (-\Delta)^{-s}\big(\rho(t)-\rho_{M(t)}\big)\cr
	&  \leq -I[\rho(t)]+P(t)^{-\delta}\big\| (-\Delta)^{-s/2}\rho(t)^r \big\|_{2}
	\big\|(-\Delta)^{-s/2}(\rho(t)-\rho_{M(t)})\big\|_{2}.\label{hdif}
	\end{align}
    Using the one dimensional entropy-entropy dissipation inequality \eqref{eq:entrdisp}
	and the relation \eqref{frac_lap_rel_entropy},
	the differential inequality~\eqref{hdif} becomes
	\[
	\frac{d}{dt} H[\rho(t)|\rho_{M(t)}]
	\leq -2H[\rho(t)|\rho_{M(t)}]
    + \sqrt{2}P(t)^{-\delta}\big\|(-\Delta)^{-s/2}\rho(t)^r\big\|_{2}
    \Big(H[\rho(t)|\rho_{M(t)}]\Big)^{1/2}.
	\]
	Applying the HLS inequality~\eqref{eq:HLS}, i.e.
	\[
	\big\|(-\Delta)^{-s/2}\rho(t)^r \big\|_{2} \leq C(s) \|\rho(t)^r\|_{\frac{2}{1+2s}}
	=C(s)\|\rho(t)\|_{\frac{2r}{1+2s}}^r\leq C(\|u_{0}\|_{1},\|u_{0}\|_{\infty},r,s),
	\]
    we get (recall that $P(t)=e^{\lambda t}$)
    \[
	\frac{d}{dt} H[\rho(t)|\rho_{M(t)}]
	\leq -2H[\rho(t)|\rho_{M(t)}]
    + C(\|u_0\|_1,\|u_0\|_\infty,r,s)e^{-\lambda \delta t}
    \Big(H[\rho(t)|\rho_{M(t)}]\Big)^{1/2}.
    \]
    If the relative entropy $H[\rho(t)|\rho_{M(t)}]$ is strictly positive on any time
    interval $(t_1,t_2)$, then
    \[
	\frac{d}{dt} \Big(H[\rho(t)|\rho_{M(t)}]\Big)^{1/2}
	\leq -\Big(H[\rho(t)|\rho_{M(t)}]\Big)^{1/2} +C(\|u_0\|_1,\|u_0\|_\infty,r,s) \exp(-\lambda \delta
	t), \qquad t \in (t_1,t_2),
	\]
	By Gronwall's inequality, for any $t \in (t_1,t_2)$,
	\[
	H[\rho(t)|\rho_{M(t)}]
	\leq \left[ e^{-(t-t_1)}\left( H[\rho(t_1)|\rho_{M(t_1)}]\right)^{1/2}
	+C(\|u_0\|_1,\|u_0\|_\infty,r,s) e^{-(t-t_1)} \int_{t_1}^t
	\exp( \tau - \lambda \delta \tau)d\tau\right]^2.
	\]
%If there is an interval $(t_{0},t_{1})\subset (0,\bar{t})$ in which $H[\rho|\rho_{M(t)}]$ is
%strictly positive and $H[\rho|\rho_{M(t)}]=0$ for $t_{0}=0$, arguing as before in the interval
%$(t_{0},t_{1})$ we find	
%\[
%	H[\rho(t)|\rho_{M(t)}]
%	\leq \left[ \frac{\sqrt{2}\kappa}{2} e^{-t} \int_{t_0}^t
%	\exp( \tau - \lambda \delta \tau)d\tau\right]^2,\quad t\in (t_0,t_1).
%\]
By choosing $t_1$ as small as possible, we have either $t_1=0$ or
$H[\rho(t_1)|\rho_{M(t_1)}]=0$, and  hence
the following exponential convergence for any time $t>0$
\begin{equation}
	H[\rho(t)|\rho_{M(t)}] \leq
C(1+t)^2\exp\big(-2\min(1,\lambda\delta)
t\big),\label{decayrelativeentr}
	\end{equation}
where $C=C(H[u_0|\rho_{M_0}],\|u_0\|_1,\|u_0\|_{\infty},r,s)$.
	
\end{proof}
%\[
%  \frac{d}{dt}\Big(H[\rho]-H[\rho_{M(t),\infty}]\Big)\leq
%-\Big(2- \exp(-\lambda\delta t) \kappa
%\Big)\big(H[\rho]-H[\rho_{M(t),\infty}]\big).
%\]

Once the exponential convergence of the relative entropy is proved,
the convergence in  other metrics can be readily available, for instance
the Wasserstein metric that is defined as
\[
  W_2(\rho_1,\rho_2) = \left( \inf_{\pi \in \prod_M(\rho_1,\rho_2)}
  \iint_{\mathbb{R}^N\times\mathbb{R}^N} |x-y|^2 \mathrm{d}\pi (x,y)\right)^{1/2}
\]
with $\prod_M(\rho_1,\rho_2)$ being the set of all nonnegative Radon measures with total mass $M_0$
on $\mathbb{R}^N\times\mathbb{R}^N$ and marginals $\rho_1$ and $\rho_2$.
Because of the displacement convexity of the entropy $H$, the following
Talagrand inequality or transportation cost inequality
\begin{equation}\label{eq:w2relent}
W_2(\rho,\rho_{M}) \leq \sqrt{ 2(H[\rho]-H[\rho_{M}])}
\end{equation}
holds true (see~\cite[Eq. (2.7)]{MR3279352}), from which the Wasserstein distance between $\rho(t)$ and $\rho_{M(t)}$
also decreases exponentially fast as in the following corollary.
\begin{cor}[Convergence in Wasserstein metric] Under the same condition as in  Theorem \ref{expdecaysmooth},
	\[
	W_2(\rho(t),\rho_{M(t)}) \leq C (1+t)\exp(-\min(1,\lambda\delta)t).
	\]
where $C$ has the same dependence as in Theorem \ref{expdecaysmooth}.
\end{cor}

%%%%%%%%%%%%%%%%%%%%%%%%%%%%%%%%%%%%%%%%%%%%%%%%%%%%%%%%%%%%%%%%%%%%%%%%%%%%%%%%%%%%%

\subsection{Exponential convergence for general weak solutions}

In the previous subsection, the solutions are assumed to be smooth
to justify all formal computations.
However, the weak solutions obtained from the limiting sequences may not possess the required regularity and the exponential convergence of the relative entropy has to proved
for the solutions of the regularised equation first, followed by a similar limiting procedure.
To start, we consider the following regularized problem of Eq.~\eqref{eq:smabs}:
\begin{equation}\label{eq:rhoregeq}
\frac{\partial \rho^\epsilon}{\partial t}
=\nabla\cdot\left[ \rho^\epsilon \nabla\left( (-\Delta)^{-s}\rho^\epsilon
+\frac{1}{2}|x|^2 + \epsilon \log \rho^\epsilon\right)\right] -
P(t)^{-\delta}(\rho^\epsilon)^r,\qquad \rho^\epsilon(x,0) = u_0(x).
\end{equation}
It is easy to see that the solutions $\rho_{\epsilon}$ enjoy the same uniform $L^{p}$ bounds
described in Proposition \ref{ro:estim}.
%We denote by $M_\epsilon(t)$ the total mass of $\rho^\epsilon(t)$ and, as usual, by $M(t)$ the total mass of $\rho(t)$.
Because of the presence of linear diffusion, an associated regularized entropy is defined as
\[
H_\epsilon[\rho] = \int \left( \frac{1}{2}\rho(-\Delta)^{-s}\rho + \frac{|x|^2}{2} \rho
+\epsilon \rho \log \rho\right)
\]
together with the entropy dissipation rate
\[
I_\epsilon[\rho] = \int \rho\big|\nabla (-\Delta)^{-s}\rho+x +\epsilon \nabla \log
\rho\big|^2.
\]
To make sure that the entropy $H_\epsilon[\rho^\epsilon(t)]$ is
finite for the solution of Eq. \eqref{eq:rhoregeq}, we first assume that the initial data has  finite second
moment, since this property propagates through any time, as we prove in the following lemma.

\begin{lem}\label{boundsecmomapr}
	If $u_{0}$ is a non-negative function in  $L^1\big(\mathbb{R}^N,(1+|x|^{2})\mathrm{d}x\big)\bigcap
	L^{\infty}\big(\mathbb{R}^N\big)$, then the solution $\rho^\epsilon(t)$ of
	Eq.~\eqref{eq:rhoregeq} stays in
	$L^{1}(\mathbb{R}^N,(1+|x|^{2})\mathrm{d}x)$ for all $t>0$ and
	\begin{equation}\label{mom_bound}
	\int |x|^{2}\rho^{\epsilon} \leq C\left(
	\|u_0\|_{1},\|u_0\|_{\infty},\int |x|^{2}u_0\,\mathrm{d}x,N,s\right).
	\end{equation}
\end{lem}

\begin{proof}
	The computations below are based on the assumption that $\rho^\epsilon$ decays
	to zero fast enough in the far field, and can be made rigorous by using a cut-off function in space,
	see for instance \cite[Lemma 2.1]{DP}.
	From the governing equation \eqref{eq:rhoregeq} for $\rho^{\epsilon}$,
	\begin{align}\label{eq:momdev2}
	\frac{d}{dt} \int \frac{|x|^{2}}{2}\rho^{\epsilon}
	&= -\int \rho^{\epsilon} (x\cdot \nabla (-\Delta)^{-s}\rho^{\epsilon} +|x|^2)
	- P(t)^{-\delta}\int \frac{|x|^{2}}{2}(\rho^{\epsilon})^{r}+ \epsilon\int \frac{|x|^{2}}{2} \triangle \rho^{\epsilon}\cr
	&\leq -\int \rho^{\epsilon} x\cdot \nabla (-\Delta)^{-s} \rho^{\epsilon}-\int |x|^{2} \rho^{\epsilon}+ \epsilon N M_{\epsilon}(t) .
	\end{align}
	By the definition $(-\Delta)^{-s}\rho^{\epsilon}(x) = C_{N,s} \int
	|x-y|^{2s-N}\rho^{\epsilon}(y)dy$ of the Riesz potential, the first
	term on the right hand of~\eqref{eq:momdev2} becomes
	\[
	-\int \rho^{\epsilon} x\cdot \nabla (-\Delta)^{-s}\rho^{\epsilon}
	=  (N-2s)C_{N,s} \iint  \rho^{\epsilon}(x)\rho^{\epsilon}(y) x\cdot(x-y)|x-y|^{2s-N-2} \mathrm{d}y\mathrm{d}x.
	\]
	The last double integral can be symmetrized by taking the average with the same expression
	when the variables $x$ and $y$ are exchanged. That is,
	\[
	-\int \rho^{\epsilon} x\cdot \nabla (-\Delta)^{-s}\rho^{\epsilon}
	=\frac{N-2s}{2} C_{N,s}\iint \rho^{\epsilon}(x)\rho^{\epsilon}(y) |x-y|^{2s-N} \mathrm{d}y\mathrm{d}x
	=\frac{N-2s}{2}\int \rho^{\epsilon} (-\Delta)^{-s}\rho^{\epsilon}.
	\]
	As a result, the rate of change of second moment can be bounded as
	\begin{align}\label{55}
	\frac{d}{dt} \int \frac{|x|^{2}}{2}\rho^{\epsilon}
	\leq (N-2s)\int \rho^{\epsilon} (-\Delta)^{-s}\rho^{\epsilon} - \int|x|^{2}\rho^{\epsilon} +N
	M_0.
	\end{align}
	By the Hardy-Littlewood-Sobolev inequality~\eqref{eq:HLS} and the uniform $L^{p}$ bounds
described in Proposition \ref{ro:estim},
	\[
	\int \rho^{\epsilon} (-\Delta)^{-s}\rho^{\epsilon} \leq  \|(-\Delta)^{-\frac{s}{2}}\rho^{\epsilon}\|^2_{2}
	\leq C(N,s) \|\rho^{\epsilon}\|_{\frac{2N}{N+2s}}^2 \leq C(\|u_0\|_{1},\|u_0\|_{\infty},N,s).
	\]
	This bound, combined with the differential inequality~\eqref{55}, leads to the desired estimate \eqref{mom_bound}.
\end{proof}

Lemma \ref{boundsecmomapr} assures that all three integrals in the entropy $H_\epsilon$ are finite along the flow of Eq.~\eqref{eq:rhoregeq}. Indeed, if  the initial data $u_{0}$ is a nonnegative function in $L^1(\mathbb{R}^N,(1+|x|^{2})\mathrm{d}x)\bigcap L^{\infty}(\mathbb{R}^N)$, then
both $\int |x|^2 \rho^\epsilon(t)$ and $\int\rho^{\epsilon}(-\Delta)^{-s}\rho^{\epsilon}$ are finite.
By Carleman type estimate (see \cite[Lemma 2.2]{Blanchet-2008}),  the  uniform in time bound of the second order moments yields $\rho^{\epsilon}(t)\log\rho^{\epsilon}(t)\in L^{1}(\mathbb{R})$ at any time $t>0$. Moreover, Lemma \ref{boundsecmomapr}  implies the  mass confinement
\[
\lim_{R\to \infty} \sup_{\epsilon>0}\int_{|x|>R}\rho_{\epsilon}(t)\mathrm{d}x\leq
\lim_{R\to \infty} \sup_{\epsilon>0}\frac{1}{R^2}
\int_{|x|>R}|x|^2\rho_{\epsilon}(t)\mathrm{d}x =  0.
\]
This gives in particular the weak $L^1$ convergence of $\rho^\epsilon(t)$ to $\rho(t)$ thus in particular the convergence of $M_{\epsilon}(t)$ to $M(t)$.
Using the lower semi-continuity of the entropy $H$ with respect to the weak*-convergence, we have
\[
H[\rho(t)] \leq \liminf_{\epsilon \to 0^+} H[\rho^\epsilon(t)]
\]
and
\[
\int |x|^{2}\rho(t)\mathrm{d}x\leq \liminf_{\epsilon\to 0^+}
\int |x|^{2}\rho^\epsilon(t)\mathrm{d}x \leq
C\left(
N,s,\|u_0\|_{\infty},\|u_0\|_{1},\int |x|^{2}u_0\,\mathrm{d}x\right).
\]
Now, arguing as in the proof of Lemma \ref{boundsecmomapr} we can show that higher order moments are also preserved, that will be needed in Section~\ref{Lpdecay}
to show the exponential convergence in $L^1$-norm.
\begin{lem}[Bounds on higher order moments] \label{lem:2nm}
	Let $n$ be a positive integer and $u_0$ be a non-negative function in $ L^{1}(\mathbb{R}^N,(1+|x|^{2n})\mathrm{d}x)\cap L^{\infty}(\mathbb{R}^N)$. Then the solution $\rho^\epsilon(t)$ to Eq.~\eqref{eq:rhoregeq} with initial condition
	$u_0$ is also in $L^{1}(\mathbb{R}^N,(1+|x|^{2n})\mathrm{d}x)$ for all $t>0$ and
	\begin{equation}\label{bound_2n_moment}
	\int |x|^{2n}\rho_{\epsilon} \leq C\left(\|u_0\|_{1},\|u_0\|_{\infty}, \int |x|^{2n}u_0,n,N,s\right).
	\end{equation}	
\end{lem}

\begin{proof}  From equation \eqref{eq:smabs} satisfied by $\rho^\epsilon$,
	\begin{align*}\label{eq:momdev}
	\frac{d}{dt} \int \frac{|x|^{2n}}{2n}\rho^\epsilon
	\leq -\int |x|^{2n-2}\rho^\epsilon x\cdot \nabla (-\Delta)^{-s} \rho^\epsilon-\int |x|^{2n} \rho^\epsilon+\epsilon N(2n-1)\int |x|^{2n-2}\rho^\epsilon.
	\end{align*}
	Using the fact that for any positive integer $n$ there exist a constant $K_{N,n}$ such that
	$(|x|^{2n-2}x-|y|^{2n-2}y)\cdot(x-y)
	\leq K_{N,n}(|x|^{2n-2}+|y|^{2n-2})|x-y|^2$, 	
	the non-local integral $-\int |x|^{2n-2}\rho^\epsilon x\cdot \nabla (-\Delta)^{-s}\rho^\epsilon$ on the right hand side can be symmetrized as
	\begin{align*}
	&\quad\frac{N-2s}{2} C_{N,s}\iint \rho^\epsilon(x)\rho^\epsilon(y) (|x|^{2n-2}x-|y|^{2n-2}y)\cdot(x-y)|x-y|^{2s-N-2} \mathrm{d}y\mathrm{d}x \cr
	&\leq \frac{(N-2s)K_{N,n}}{2}\int |x|^{2n-2}\rho^\epsilon (-\Delta)^{-s}\rho^\epsilon.
	%	=  (N-2s)C_{N,s} \iint |x|^{2n-2} \rho(x)\rho(y) x\cdot(x-y)|x-y|^{2s-N-2} dydx.
	\end{align*}
	This integral can be further bounded as
	\begin{align*}
	\frac{(N-2s)K_{N,n}}{2}\int |x|^{2n-2}\rho^\epsilon (-\Delta)^{-s}\rho^\epsilon
	&\leq \frac{1}{2}\int |x|^{2n}\rho^\epsilon + \tilde{K}(N,s,n)\int \rho^\epsilon |(-\Delta)^{-s}\rho^\epsilon|^n \cr
	&\leq \frac{1}{2}\int |x|^{2n}\rho^\epsilon + C(\|u_0\|_1,\|u_0\|_\infty,n,N,s),
	\end{align*}
	where the $L^\infty$ bound of $\rho^\epsilon$ and the Hardy-Littlewood-Sobolev
	inequality is used in the last step.
	As a result, 	the rate of change of $2n$-th moment becomes
	\begin{equation}\label{higherordermom}
	\frac{d}{dt} \int \frac{|x|^{2n}}{2n}\rho^\epsilon
	\leq  C(\|u_0\|_1,\|u_0\|_\infty,n,N,s)- \frac{1}{2}\int|x|^{2n}\rho_{\epsilon}+\epsilon N(2n-1)\int |x|^{2n-2}\rho^\epsilon.
	\end{equation}
	%	On the other hand, if $u_0\in L^1(\mathbb{R}^N,(1+|x|^{2n}))$, then all moments of $u_0$ (up to order $2n$) are bounded.
	By induction on $n$,  the differential inequality~\eqref{higherordermom} implies that all moments of the solution $\rho^\epsilon$ up to order $2n$ are bounded as well.
	% i.e.,
	%	\begin{equation*}\label{bound_2n_moment}
	%	\int |x|^{2n}\rho_{\epsilon} \leq C\left(n,N,s,\|u_0\|_{\infty},\|u_0\|_{1}, \int |x|^{2n}u_0\right).
	%	\end{equation*}
\end{proof}

%\begin{remark}\label{n-dependence}
%If $N=1$ the constant in \eqref{bound_2n_moment} is such that as $n$-function $C=C(n)\sim \frac{6^n}{n}$ (up a constant) as $n\rightarrow\infty$. Indeed in the previous proof $K_{1,n}=\frac{3^{n-1}}{2}$. It is true for $n=1$ and we can proceed for induction on $n$.  Moreover we used the following version of Young inequality $ab\leq\varepsilon a^p+C(\varepsilon)b^{p'}$ with $C(\varepsilon)=\frac{p^\frac{p-2}{p-1}}{\varepsilon^\frac{1}{p-1}(p-1)}$. Then in \eqref{higherordermom} the constant $C(n)\sim6^n$ (up a constant) as $n\rightarrow\infty$.
%\end{remark}

It is straightforward to show that
$M_\epsilon(t)\leq M_{\epsilon}(0)=\|u_{0}\|_{1}=M_{0}$. We define the Barenblatt profile
$\rho_{M_\epsilon(t)}^\epsilon$ as the minimizer of $H_\epsilon[\rho]$
over the set of admissible functions\footnote{Here the superscript $\epsilon$ is used to distinguish it from $\rho_{M}$, the minimizer of $H[\rho]$ with total mass $M$.}
\[
\mathcal{Y}_{M_\epsilon(t)}=\left\{\eta\in
L_{+}^{1}(\mathbb{R}^N):\|\eta\|_{1}=M_\epsilon(t),\,\int|x|^{2}\eta(x)\mathrm{d}x<\infty\right\}.
\]
This minimizer $\rho_{M_\epsilon(t)}^\epsilon$ exists and is unique from the convexity of the functional $H_\epsilon$ (see \cite{MR3262506}).
Moreover, $\rho_{M_\epsilon(t)}^\epsilon$ satisfies the Euler-Lagrange equation
associated with $H_\epsilon$, that is
\begin{equation}
\label{ELeqapp}
(-\Delta)^{-s}\rho^{\epsilon}_{M_\epsilon(t)} + \frac{1}{2}|x|^2 + \epsilon (\log
\rho^{\epsilon}_{M_\epsilon(t)}+1)= \xi_{\epsilon}(t),
\end{equation}
for some constant $\xi_\epsilon(t)$.
First,  the counterpart of Lemma~\ref{lem:hsnorm}  relating  $H^{-s}(\mathbb{R})$
norm the relative entropy still hold for the regularised entropy.
\begin{lem}\label{AppLemma3.3} Let $\eta$ be a non-negative function with total mass $M_\epsilon(t)$
	such that $H_\epsilon[\eta]$ is finite,
	then $\|(-\Delta)^{-s/2}(\eta-\rho^\epsilon_{M_\epsilon(t)})\|_{2}^2
	\leq 2\big( H_\epsilon[\eta]-H_\epsilon[\rho_{M_\epsilon(t)}^\epsilon]\big)$.
\end{lem}
\begin{proof}
	Since $\rho_{M_\epsilon(t)}^\epsilon$ has the same mass $M_\epsilon(t)$ as $\eta$,
	from the characterization~\eqref{ELeqapp}, we get
	\[
	0 = \int \xi_\epsilon(t)\Big( \eta -\rho^\epsilon_{M_\epsilon(t)}\Big)
	= \int \left(\eta-\rho_{M_\epsilon(t)}^\epsilon\right)
	\left( (-\Delta)^{-s}\rho^{\epsilon}_{M_\epsilon(t)} + \frac{1}{2}|x|^2 + \epsilon (\log
	\rho^{\epsilon}_{M_\epsilon(t)}+1)\right).
	\]
	This identity can be rearranged to obtain the relation
	\[
	H_\epsilon[\eta]-H_\epsilon[\rho^{\epsilon}_{M_\epsilon(t)}]
	=\frac{1}{2}\|(-\Delta)^{-s/2}(\eta-\rho^{\epsilon}_{M_\epsilon(t)})\|^{2}_{2}
	+\epsilon\int\eta\log\frac{\eta}{\rho^{\epsilon}_{M_\epsilon(t)}}\mathrm{d}x.
	\]
	Applying Jensen's inequality (with then measure $d\mu =
	\frac{\rho_{M_\epsilon(t)}^\epsilon}{M_\epsilon(t)}\mathrm{d}x$),
	\[
	\int\eta\log\frac{\eta}{\rho^{\epsilon}_{M_\epsilon(t)}}\mathrm{d}x
	=M_\epsilon(t)\int\frac{\eta}{\rho^{\epsilon}_{M_\epsilon(t)}}
	\log\frac{\eta}{\rho^{\epsilon}_{M_\epsilon(t)}}d\mu(t)
	\geq M_\epsilon(t)
	\left(\int \frac{\eta}{\rho_{M_\epsilon(t)}^\epsilon}d\mu \right)
	\log\left(\int \frac{\eta}{\rho_{M_\epsilon(t)}^\epsilon}d\mu \right)
	=0.
	\]
	Therefore, the desired inequality holds.
\end{proof}
Although $\rho_{M_\epsilon(t)}^\epsilon$
is not expected to have explicit expressions as $\rho_{M}$ in \eqref{eq:profile},
we can still get uniform bounds with respect to $\epsilon$ and $t>0$.
\begin{prop}\label{Propregminim} The family of minimizers $\rho_{M_\epsilon(t)}^\epsilon$ is smooth and uniformly bounded in $L^{\infty}(\mathbb{R}^{N})$ with respect to $\epsilon$.
\end{prop}
\begin{proof}
	We first show that the minimizers are smooth for any $\epsilon$ small enough, and then show the bounds are uniform.
	%For finite $\epsilon$, we can also show bounds and regularity of
	%the family of minimizers $\big\{\rho^{\epsilon}_{M_\epsilon(t)}\big\}$.
	Multiplying both sides
	of the Euler-Lagrange equation~\eqref{ELeqapp} with $\rho_{M_\epsilon(t)}^\epsilon$ and integrating
	on the whole space, we get
	\begin{align*}
	(\xi_\epsilon(t)-\epsilon)M_\epsilon(t)
	=\frac{1}{2}\int \rho_{M_\epsilon(t)}^\epsilon(-\Delta)^{-s}
	\rho_{M_\epsilon(t)}^\epsilon + H_\epsilon[\rho_{M_\epsilon(t)}^\epsilon]
	\leq H_\epsilon[\rho_{M_\epsilon(t)}^\epsilon]
	\leq H_\epsilon[\rho_{M_\epsilon(t)}],
	\end{align*}
	where the fact that $\rho_{M_\epsilon(t)}^\epsilon$ is
	the minimizer of $H_\epsilon$ is used in the last inequality.
	Since $M_\epsilon(t)$ converges to $M(t)>0$, $H_\epsilon[\rho_{M_\epsilon(t)}]/M_\epsilon(t)$
	is uniformly bounded (for $\epsilon$ small), and so is $\xi_\epsilon-\epsilon$.
	Next again from the Euler-Langrange equation \eqref{ELeqapp},
	\begin{equation}
	 \rho^{\epsilon}_{M_\epsilon(t)}=\exp\left(\frac{1}{\epsilon}\left(\xi_{\epsilon}(t)-(-\Delta)^{-s}\rho^{\epsilon}_{M_\epsilon(t)}-\frac{|x|^{2}}{2}\right)-1\right)\leq
	\exp\left(\frac{\xi_{\epsilon}(t)}{\epsilon}\right)\label{ELeqexp},
	\end{equation}
	which implies that $\rho^{\epsilon}_{M_\epsilon(t)}\in L^{\infty}(\mathbb{R})$
	for any finite $\epsilon>0$.

	To prove that $\rho^{\epsilon}_{M_\epsilon(t)}$ is smooth,  we follow the main ideas from \cite[Theorem 10]{CHMV} for a similar calculation for the case $s=1$.
	Indeed, using the interior regularity result for the fractional Laplacian stated in\cite[Theorem 1.1 and Corollary 3.5]{Ros-Oton} (see also \cite[Proposition
	5.2]{CafSting}) we have $(-\Delta)^{-s}\rho^{\epsilon}_{M_\epsilon(t)}\in L^{\infty}(\mathbb{R})$
	and
	\begin{equation*}
	\|(-\Delta)^{-s}\rho^{\epsilon}_{M_\epsilon(t)}\|_{C^{0,2s}}\leq
	C\left(\|(-\Delta)^{-s}\rho^{\epsilon}_{M_\epsilon(t)}\|_{\infty}+\|\rho^{\epsilon}_{M_\epsilon(t)}\|_{\infty}\right)\, .
	\end{equation*}
	then from \eqref{ELeqexp} we obtain $\rho^{\epsilon}_{M_\epsilon(t)}\in C^{0,2s}$. Then we apply again the above mentioned regularity estimate of
	\cite{Ros-Oton}, saying that for $\alpha > 0$ such that $\alpha+2s$ not an integer,
	\begin{equation*}\label{HR}
	\|(-\Delta)^{-s}\rho^{\epsilon}_{M_\epsilon(t)}\|_{C^{0,\alpha+2s}}\le
	C\left(\|(-\Delta)^{-s}\rho^{\epsilon}_{M_\epsilon(t)}\|_{\infty}+\|\rho^{\epsilon}_{M_\epsilon(t)}\|_{C^{0,\alpha}}\right)\, .
	\end{equation*}
	to find that $(-\Delta)^{-s}\rho^{\epsilon}_{M_\epsilon(t)}\in C^{\gamma}$ for any $\gamma<4s$. Then by \eqref{ELeqexp} we have
	$\rho^{\epsilon}_{M_\epsilon(t)}\in C^{\gamma}$ for any $\gamma<4s$. Arguing iteratively, any order $\ell$ of differentiability for
	$(-\Delta)^{-s}\rho^{\epsilon}_{M_\epsilon(t)}$ (and then for $\rho^{\epsilon}_{M_\epsilon(t)}$) can be reached and hence $\rho^{\epsilon}_{M_\epsilon(t)}\in
	C^{\infty}$.
	we finally obtain $\rho^{\epsilon}_{M_\epsilon(t)}$ is actually $C^{\infty}$.\\[0.2pt]
	With the regularity properties we have just proved at hand, it is quite standard to obtain a uniform bound of $\rho^{\epsilon}_{M_\epsilon(t)}$ in time and
	$\epsilon>0$. Indeed, from equation \eqref{ELeqapp} we have
	\begin{align*}
	0 &= \int \nabla (\rho^{\epsilon}_{M_\epsilon(t)})^p \cdot
	\nabla\left[
	(-\Delta)^{-s}\rho^{\epsilon}_{M_\epsilon(t)} + \frac{1}{2}|x|^2 + \epsilon (\log
	\rho^{\epsilon}_{M_\epsilon(t)}+1)
	\right] \cr
	&\geq \int \nabla (\rho^{\epsilon}_{M_\epsilon(t)})^p \cdot
	\nabla (-\Delta)^{-s}  \rho^{\epsilon}_{M_\epsilon(t)}
	- N\int \big(\rho^{\epsilon}_{M_\epsilon(t)}\big)^p \cr
	& \geq \frac{4p}{(p+1)^2}\int\left|(-\Delta)^{\frac{1-s}{2}}(\rho^{\epsilon}_{M_\epsilon(t)})^{\frac{p+1}{2}}\right|^{2}\mathrm{d}x
	- N\int \big(\rho^{\epsilon}_{M_\epsilon(t)}\big)^p,
	\end{align*}
	where the Stroock-Varopoulos inequality \eqref{eq:svineq} is used. Since a direct application of Nash-Gagliardo-Nirenberg inequality \eqref{GN}
	does not give bounds on $\|\rho^{\epsilon}_{M_\epsilon(t)}\|_\infty$, a classical Moser type iteration is used first for different norms.
	%we find
	%\[
	%\int (\rho^{\epsilon}_{M_\epsilon(t)})^{p}dx\geq \left(\frac{4p}{(p+1)^2}\right)^{1/2}\|(-\Delta)^{\frac{1-s}{2}}w\|_{2}\|w\|_{2p/(p+1)}^{p/(p+1)}.
	%\]
	%We follow now \cite[Proposition 5.4]{VazVol2}.
	Then we apply the Nash-Gagliardo-Nirenberg inequality \eqref{GN}  with the choices
	\[
	\alpha=1-s,\quad l=\frac{2p}{p+1},\quad \theta=\frac{l}{2}=\frac{p}{p+1},\quad
	m=\frac{N(2p+1)}{(N-1+s)(p+1)}
	\]
	and $v=\big(\rho^{\epsilon}_{M_\epsilon(t)}\big)^{(p+1)/2}$, in order to obtain
	\begin{align*}
	\int (\rho^{\epsilon}_{M_\epsilon(t)})^{p}\mathrm{d}x&\geq C(N,s,p)\|w\|_{m}^{\theta+1}
	=C(N,s,p)\|\rho^{\epsilon}_{M_\epsilon(t)}\|_{\mathsf{q}}^{(2p+1)/2}
	\end{align*}
	with
	$
	\mathsf{q}=\frac{N(2p+1)}{2(N-1+s)}>p
	$.
	Then we finally have
	\begin{equation}
	\|\rho^{\epsilon}_{M_\epsilon(t)}\|_{p}^{2p}\geq \mathsf{C} \|\rho^{\epsilon}_{M_\epsilon(t)}\|_{\mathsf{q}}^{2p+1},\label{firstiter}
	\end{equation}
	where the constant $\mathsf{C}=\mathsf{C}(N,s,p)$ verifies the asymptotic
	\[
	\mathsf{C}=\mathsf{C}(p,s)\sim \frac{C(s,N)}{p^2}\quad as\,\, p\rightarrow\infty.
	\]
	Now we can use~\eqref{firstiter} to iterate between different norms to get $\|\rho^{\epsilon}_{M_\epsilon(t)}\|_{\infty}$. We set $p_{0}=p$ and
	\[
	p_{k+1}=\sigma\left(p_{k}+\frac{1}{2}\right),\quad \sigma=\frac{N}{N-1+s}.
	\]
	We notice that, by our assumptions, the sequence $\left\{p_{k}\right\}$ is defined through
	\[
	p_{k}=A(\sigma^{k}-1)+p,\quad A=\frac{N}{2(1-s)}+p>0
	\]
	so that $p_{k}<p_{k+1}$ and $\lim_{k\rightarrow\infty}p_{k}=+\infty$. Then inequality \eqref{firstiter} provides
	\begin{equation*}
	\mathsf{C}(N,s,p_{k}) \|\rho^{\epsilon}_{M_\epsilon(t)}\|_{p_{k+1}}^{2p_{k}+1}\leq \|\rho^{\epsilon}_{M_\epsilon(t)}\|^{2p_{k}}_{p_{k}}
	\end{equation*}
	from which (recalling that $p_{k}\rightarrow\infty$),
	\begin{equation*}
	\|\rho^{\epsilon}_{M_\epsilon(t)}\|_{p_{k+1}}\leq \mathsf{C}(p_{k},s)^{-\frac{1}{2p_{k}+1}}
	\|\rho^{\epsilon}_{M_\epsilon(t)}\|^{\sigma\frac{p_{k}}{p_{k+1}}}_{p_{k}}.
	\end{equation*}
	Since
	\[
	\lim_{k\rightarrow\infty}\mathsf{C}(p_{k},s)^{-\frac{1}{2p_{k}+1}}=1,
	\]
	setting $U_k=\|\rho^{\epsilon}_{M_\epsilon(t)}\|_{p_{k}}$, we can iterate as in \cite[Proposition 5.4]{VazVol2} and we can pass the limit and obtain the following bound of the $L^{\infty}$ norm of
	$\rho^{\epsilon}_{M_\epsilon(t)}$ in terms of its $L^{p}$ norm:
	\begin{equation*}
	\|\rho^{\epsilon}_{M_\epsilon(t)}\|_{\infty}\leq C\|\rho^{\epsilon}_{M_\epsilon(t)}\|_{p}^{\frac{2p(1-s)}{1+2p(1-s)}}.\label{LpLinfty}
	\end{equation*}
	Since this inequality holds for all $p$, we can choose $p=1$, thus
	\begin{align*}
	\|\rho^{\epsilon}_{M_\epsilon(t)}\|_{\infty}\leq C(N,s)M_\epsilon(t)^{\frac{2(1-s)}{1+2(1-s)}}\leq C(N,s,M_{0}),
	\end{align*}
	hence the proof follows.
\end{proof}

We now focus on the convergence of the solution $\rho^\epsilon(t)$ of the regularised equation~\eqref{eq:rhoregeq} towards the time-dependent profile
$\rho^\epsilon_{M_\epsilon(t)}$,
to recover the same rate as the regularization constant $\epsilon$ goes to zero.
Arguing similarly to the smooth case in one dimension, we obtain
\[
\frac{\mathrm{d}}{\mathrm{d}t} H_\epsilon[\rho^{\epsilon}]=-I_{\epsilon}[\rho^{\epsilon}]-P(t)^{-\delta}
\int (\rho^\epsilon)^r \left[(-\Delta)^{-s}\rho^{\epsilon}+\frac{|x|^{2}}{2}+\epsilon(1+ \log
\rho^{\epsilon}))\right],
\]
and
\begin{align*}
\frac{\mathrm{d}}{\mathrm{d}t} H_\epsilon[\rho^{\epsilon}_{M_\epsilon(t)}]
&=\int\left((-\Delta)^{-s}\rho^{\epsilon}_{M_\epsilon(t)} + \frac{|x|^2}{2}
+\epsilon(\log \rho^\epsilon_{M_\epsilon(t)} +1)\right)
\frac{\partial\rho^{\epsilon}_{M_\epsilon(t)}}{\partial t}\\
&=\xi_\epsilon(t)\int \frac{\partial\rho^{\epsilon}_{M_\epsilon(t)}}{\partial
	t}\mathrm{d}x=\xi_\epsilon(t)\frac{\mathrm{d}}{\mathrm{d}t}M_\epsilon(t),
\end{align*}
where the Euler-Lagrange equation~\eqref{ELeqapp} satisfied by the regularised profile $\rho^{\epsilon}_{M_\epsilon(t)}$ is used.
From $\frac{\mathrm{d}}{\mathrm{d}t}M_\epsilon(t) = -P(t)^{-\delta}\int (\rho^\epsilon)^r$, $\frac{d}{dt} H_\epsilon[\rho^{\epsilon}_{M_\epsilon(t)}]$
can be written as
\begin{align*}
\frac{d}{dt} H_\epsilon[\rho^{\epsilon}_{M_\epsilon(t)}]
&=-\xi_\epsilon(t)P(t)^{-\delta}\int (\rho^\epsilon)^r \mathrm{d}x \cr
&=-P(t)^{-\delta}\int\left((-\Delta)^{-s}\rho^{\epsilon}_{M(t)} + \frac{1}{2}|x|^2 + \epsilon (\log \rho^{\epsilon}_{M(t)}+1)\right) (\rho^\epsilon)^r \mathrm{d}x.
\end{align*}
Therefore the time derivative of the relative entropy becomes
\begin{align}
\frac{d}{dt} \big(
H_\epsilon[\rho^{\epsilon}]-H_\epsilon[\rho^{\epsilon}_{M_\epsilon(t)}]\big)=
-I_\epsilon[\rho^\epsilon] - P(t)^{-\delta}\int (\rho^\epsilon)^r \big[     (-\Delta)^{-s}(\rho^{\epsilon}-\rho^{\epsilon}_{M_\epsilon(t)})
+\epsilon \log (\rho^{\epsilon}/\rho^{\epsilon}_{M_\epsilon(t)})\big].\label{evolapproxentropy}
\end{align}
On the other hand, the logarithmic term in~\eqref{evolapproxentropy} can be written as
\begin{align*}
\int(\rho^\epsilon)^{r}\log\frac{\rho^\epsilon}{\rho^{\epsilon}_{M_\epsilon(t)}}\mathrm{d}x
=\frac{\|\rho^{\epsilon}_{M_\epsilon(t)}\|_{r}^{r}}{r}
\int\left(\frac{\rho^\epsilon}{\rho^{\epsilon}_{M_\epsilon(t)}}\right)^{r}
\log\left(\frac{\rho^\epsilon}{\rho^{\epsilon}_{M_\epsilon(t)}}\right)^{r}d\mu,
\qquad
d\mu=\frac{(\rho^{\epsilon}_{M(t)})^{r}}{\|\rho^{\epsilon}_{M(t)}\|_{r}^{r}}\mathrm{d}x.
\end{align*}
By Jensen's inequality again we have
\begin{align*}
\int(\rho^\epsilon)^{r}\log\frac{\rho^\epsilon}{\rho^{\epsilon}_{M_\epsilon(t)}}\mathrm{d}x
&\geq \frac{\|\rho^{\epsilon}_{M_\epsilon(t)}\|_{r}^{r}}{r}
\left[\int\left(\frac{\rho^\epsilon}{\rho^{\epsilon}_{M_\epsilon(t)}}\right)^{r}\mathrm{d}\mu\right] \log\left[\int\left(\frac{\rho^\epsilon}
{\rho^{\epsilon}_{M_\epsilon(t)}}\right)^{r}\mathrm{d}\mu\right]\cr
&\geq -\frac{\|\rho^{\epsilon}_{M_\epsilon(t)}\|_{r}^{r}}{r}\frac{1}{e},
\end{align*}
where the fact that $f(z)=z \log z  \geq -1/e$ for $z \geq 0$ is used
in the last step.
As a result, by Proposition \ref{Propregminim} we find
\[
-\epsilon \int
(\rho^\epsilon)^{r}\log\frac{\rho^\epsilon}{\rho^{\epsilon}_{M_\epsilon(t)}}\mathrm{d}x\leq\epsilon
C \|\rho^{\epsilon}_{M_\epsilon(t)}\|_{r}^{r}
\leq\epsilon C_1(\|u_{0}\|_1,r,s).
\]
Now we notice that (see \cite[Proposition 2.4]{MR3279352} for more details)
the entropy-entropy dissipation inequality is still valid for $H_\epsilon$, i.e.,
\begin{equation*}
H_\epsilon[\rho^\epsilon] -
H_\epsilon[\rho^{\epsilon}_{M_{\epsilon}(t)}] \leq \frac{1}{2}I_\epsilon[\rho^\epsilon].
\end{equation*}
Consequently we get the following self-contained differential inequality from \eqref{evolapproxentropy},
\begin{multline}
\frac{d}{dt}\label{diffineqeps}
\big( H_\epsilon[\rho^{\epsilon}]-H_\epsilon[\rho^{\epsilon}_{M_{\epsilon}(t)}]\big)
\leq
-2\big( H_\epsilon[\rho^{\epsilon}]-H_\epsilon[\rho^{\epsilon}_{M(t)}]\big) \cr
+C(\|u_{0}\|_1,\|u_{0}\|_{\infty},r,s)P(t)^{-\delta}
\big( H_\epsilon[\rho^{\epsilon}]-H_\epsilon[\rho^{\epsilon}_{M(t)}\big)^{1/2}
+\epsilon C_1(\|u_{0}\|_1,r,s) P(t)^{-\delta}.
\end{multline}
Now we can prove Theorem~\ref{thm:abs} about the exponential convergence of
relative entropy of general weak solutions, by taking the limit as $\epsilon$
goes to zero.

\begin{proof}[Proof of Theorem~\ref{thm:abs}]
	If we set $f_\epsilon(t) =
	H_\epsilon[\rho^{\epsilon}(t)]-H_\epsilon[\rho^{\epsilon}_{M_{\epsilon}(t)}] $, then \eqref{diffineqeps} implies that  $f_\epsilon$  satisfies the differential inequality
	\[
	f'_\epsilon(t) \leq -2f +C_1e^{-\lambda \delta t}f^{1/2} + \epsilon C_2 e^{-\lambda \delta t}.
	\]
	From Lemma~\ref{lem:fdeq} below, we have
	\[
	\liminf_{\epsilon \to 0^+} \big( H_\epsilon[\rho^{\epsilon}(t)]-H_\epsilon[\rho^{\epsilon}_{M_{\epsilon}(t)}]\big)
	\leq C(1+t)^2e^{-2\min(1,\lambda\delta)}
	\]
	and our aim now is to show the above limit on the left hand side is larger than $H[\rho(t)]-H[\rho_{M(t)}]$. Since $\rho^\epsilon$ is uniformly bounded
	in $L^p(\mathbb{R})$ with finite second moment, $|\int \rho^\epsilon \log
	\rho^\epsilon| $ is also uniformly bounded. From the lower semi-continuity
	of the entropy $H[\rho]$ (\cite[Theorem 3.1]{MR3279352}), when $\epsilon$ goes to zero,  $\rho^\epsilon(t)$ 	converges to $\rho(t)$ in weak $L^1$, and
	\[
	\liminf_{\epsilon \to 0^+} H_\epsilon[\rho^\epsilon(t)]
	\geq 	\liminf_{\epsilon \to 0^+} H[\rho^\epsilon(t)]
	= H[\rho(t)].
	\]
	For the other term $H_\epsilon[\rho^{\epsilon}_{M_{\epsilon}(t)}]$,
	we can show that the limit as $\epsilon$ goes to zero actually exist.
	If we define
	$$\mathcal{H}[\eta] = \pi M_\epsilon(t)^{-2}\int |x|^2\eta + \int \eta\log \eta$$
	for any function $\eta$ in $L^1(\mathbb{R},(1+|x|^2)\mathrm{d}x)$ such that $\int \eta = M_\epsilon(t)$, then
	$\mathcal{H}[\eta]$ is non-negative and $\mathcal{H}[\eta] =0$
	only when $\eta(x) = \exp(-\pi |x|^2/M_\epsilon(t)^2)$.
	If  $\epsilon$ is restricted to be less than $\frac{1}{2\pi}M_{\epsilon}(t)$,
	\begin{align*}
	H_\epsilon[\rho^\epsilon_{M_\epsilon(t)}] &=
	\epsilon \mathcal{H}[\rho^\epsilon_{M_\epsilon(t)}] +
	(1-2\epsilon\pi M_\epsilon(t)^{-2})H[\rho_{M_\epsilon(t)}^\epsilon]
	+\epsilon\pi M_\epsilon(t)^{-2} \int \rho_{M_\epsilon(t)}^\epsilon (-\Delta)^{-s}\rho_{M_\epsilon(t)}^\epsilon \cr
	&\geq \left(1-2\epsilon\pi M_\epsilon(t)^{-2}\right)H[\rho_{M_\epsilon(t)}],
	\end{align*}
	which implies that
	\[
	\liminf_{\epsilon\to 0}  H_\epsilon[\rho^\epsilon_{M_\epsilon(t)}] \geq \lim_{\epsilon\to 0}
	\left(1-2\epsilon\pi M_\epsilon(t)^{-2}\right)H[\rho_{M_\epsilon(t)}] = H[\rho_{M(t)}].
	\]
	Here in the last step the limit exists
	because the profile $\rho_{M_\epsilon(t)}$ converges to $\rho_{M(t)}$ and due to the boundedness $\int \rho_{M_\epsilon(t)} \log \rho_{M_\epsilon(t)}$. Moreover, due to the minimality of $\rho_{M_\epsilon(t)}$ we have
	\[
	H_\epsilon[\rho^\epsilon_{M_\epsilon(t)}]\leq H_\epsilon[\rho_{M_\epsilon(t)}]=H[\rho_{M_\epsilon(t)}]+\epsilon\int \rho_{M_\epsilon(t)} \log \rho_{M_\epsilon(t)}\mathrm{d}x
	\]
	then
	\[
	\limsup_{\epsilon\to 0}  H_\epsilon[\rho^\epsilon_{M_\epsilon(t)}] \leq H[\rho_{M(t)}]
	\]
	which finally gives
	\[
	\lim_{\epsilon\to 0}  H_\epsilon[\rho^\epsilon_{M_\epsilon(t)}] = H[\rho_{M(t)}].
	\]
	Putting all these together, we get the exponential convergence~\eqref{exponentialconv},
	which is exactly~\eqref{eq:relenabs} via the
	change of variables~\eqref{eq:simivar}. The convergence of $u(\tau)-u_{M(\tau)}$
	in the $H^{-s}(\mathbb{R})$ norm is then a direct consequence of~\eqref{frac_lap_rel_entropy}.
\end{proof}

In the previous theorem we used the following technical lemma.
\begin{lem} \label{lem:fdeq}
	Let $f_\epsilon(t)$ be a family of non-negative functions defined on $t \in [0,\infty)$ satisfying the differential inequalities $f'_\epsilon(t)\leq F(f_\epsilon(t),t,\epsilon)$.  Here  $\epsilon \in (0,\bar{\epsilon}]$
	for some $\bar{\epsilon}>0$ and
	\[
	F(f,t,\epsilon) = -2f+C_1e^{-\lambda\delta t} f^{\frac{1}{2}}+\epsilon C_2 e^{-\lambda\delta t}
	\]
	for some  positive constants $C_1$, $C_2$ and $\lambda \delta>0$. Then the following statements hold.
	\begin{enumerate}[(a)]
		\item The family of functions $f_\epsilon$ satisfies the bound
		\[	f_\epsilon(t) \leq e^{-2t}f_\epsilon(0)+C_4(1+t)e^{-\min(2,\lambda \delta)t}\]	
		where $$C_4 = C_1C_3^{1/2}+\bar{\epsilon}C_2\,\quad\, and\,\quad
		C_3=\max\left\{\sup_{\epsilon \in (0,\bar{\epsilon}]}
		f_\epsilon(0),C_1^2,\bar{\epsilon}C_2\right\};$$
		\item we have
		$$
		\liminf_{\epsilon \to 0^+} f_\epsilon(t) \leq C(1+t)^2e^{-2\min(1,\lambda\delta)t}.
		$$
		for some positive constant $C$ depending on $\bar{\epsilon},C_1,C_2$ and $\lim_{\epsilon\to 0^+} f_\epsilon(0).$
	\end{enumerate}
\end{lem}
\begin{proof}
	
	(a) First by the classical comparison principle of ODEs between $f_\epsilon(t)$ and the constant $C_3$ defined above, $F(C_3,t,\epsilon) \leq 0$ and hence
	$f_\epsilon(t)$ is uniformly bounded above by $C_3$. Therefore, $f_\epsilon(t)$ satisfies the simpler differential inequality
	\[
	f'_\epsilon  \leq -2f_\epsilon + C_4 e^{-\lambda \delta t}.
	\]
	The desired inequality is obtained by integrating the equivalent inequality
	$\frac{d}{dt}(e^{2t}f_\epsilon(t))\leq C_4e^{(2-\lambda\delta)t}$.

	(b) In addition to the comparison principle of differential inequalities,
	continuous dependence of solutions of ODEs on the parameter $\epsilon$ will also be used, where the main barrier of non-Lipschitz continuity of $F(f,t,\epsilon)$ at $f=0$ in applying these techniques is considered separately.
	
	First define $g_\epsilon(t)$ to be the solution of the ODE $g'_\epsilon(t) =F(g_\epsilon(t),t,\epsilon)$ with $g_\epsilon(0)=f_\epsilon(0)$. 	
	Then from $g'_\epsilon(t)\geq -2g_\epsilon(t)$, we have  $g_\epsilon(t) \geq g_\epsilon(0)e^{-2t}$ on
	any finite interval $[0,T]$ and $F(g_\epsilon,t,\epsilon)$
	is now Lipschitz on $(g_\epsilon(0),\infty)\times (0,\infty) \times (0,\bar{\epsilon})$. Then by the continuous dependence of ODEs, the limit
	$g(t) = \lim_{\epsilon \to 0^+} g_\epsilon(t)$
	exists, and $g(t)$ satisfies $g'(t)=F(g(t),t,0)$ with the initial condition $g(0)=\liminf_{\epsilon\to 0^+} f_\epsilon(0)$. Moreover, $g(t) \leq
	C(1+t)^2e^{-2\min(1,\lambda\delta)t}$ by the same argument as in the proof
	of Theorem~\ref{expdecaysmooth}.
	
	Next we apply the comparison principle between $f_\epsilon$ and $g_\epsilon$. If $f_\epsilon(t) < f_\epsilon(0)e^{-2t}$, we get
	$f_\epsilon(t) \leq g_\epsilon(t)$; otherwise if $f_\epsilon(t) \geq f_\epsilon(0)e^{-2t}$ on some interval $[t_0,t_1]$, the comparison principle between $f_\epsilon(t)$ and $g_\epsilon(t)$ still applies, where the initial condition $f_\epsilon(t)\leq g_\epsilon(t)$ can be enforced at $t=t_0$ by choosing $t_0$ as small as possible. In either case,
	the relation $f_\epsilon(t)\leq g_\epsilon(t)$ is satisfied on any finite interval $[0,T]$, on which
	\[
	\liminf_{\epsilon \to 0^+} f_\epsilon(t) \leq \lim_{\epsilon\to 0^+} g_\epsilon (t) = g(t).
	\]
\end{proof}

%%%%%%%%%%%%%%%%%%%%%%%%%%%%%%%%%%%%%%%%%%%%%%%%%%%%%%%%%%%%%%%%%%%%%%%%%%%%
\subsection{Comments on $L^2$ and $L^1$ decays}\label{Lpdecay}
Once the exponential convergence in the relative entropy $H[\rho(t)|\rho_{M(t)}]$
(or $H[u(\tau)|u_{M(\tau)}]$) is established, the convergence in $\dot{H}^{-s}(\mathbb{R})$ norm or in the Wasserstein metric $W_2$  is straightforward, using the inequalities~\eqref{frac_lap_rel_entropy} and~\eqref{eq:w2relent}. However, the convergence in other common norms
like $L^1(\mathbb{R})$ or $L^2(\mathbb{R})$ is less obvious. In the case of classical heat equation,
the convergence in $L^1(\mathbb{R})$ norm can be derived from the relative Boltzmann entropy, using the well-known Csisz\'{a}r-Kullback inequality. For the entropy $H[\rho]$ used here,  the convergence in $L^1(\mathbb{R})$ or $L^2(\mathbb{R})$ norm can not be established directly using similar inequalities, and  has to rely on interpolation lemma like below~\cite[Theorem 3.4]{MR3279352}, with additional assumptions on the H\"{o}lder regularity of the solution $\rho(t)$.
\begin{lem}[Interpolation between norms]\label{interpolemm} Let $0<\mathsf{a}\leq1$, $0<s<N/2$ and $0<\mathsf{a}<\alpha/2$. There exists
	a constant $C$ depending on $N$, $s$ and $\mathsf{a}$ only, such that
	\[
	\|u\|_2 \leq C \|(-\Delta)^{-s/2}u\|_2^{\sigma_1} [u]_\alpha^{\sigma_2}
	\|u\|_1^{\sigma_3}
	\]
	for any function $u \in L^1(\mathbb{R}^N)\bigcap C^\alpha(\mathbb{R}^N)$ with
	\[
	\sigma_{1}=\frac{\mathsf{a}}{\mathsf{a}+s},\quad
	\sigma_2 = \frac{s(N+2\mathsf{a})}{2(N+\alpha)(s+\mathsf{a})},\quad
	\sigma_3 = \frac{s(N+2\alpha-2\mathsf{a})}{2(N+\alpha)(s+\mathsf{a})},
	%\quad  \mathsf{a}<2\min\left\{\alpha, 1-s\right\}.
	\]
where $[\cdot]_\alpha$ denotes the H\"{o}lder seminorm.
\end{lem}

The explicit expression~\eqref{eq:profile} of $\rho_{M(t)}$ implies that
$\rho_{M(t)}$ is $(1-s)$-H\"{o}lder continuous with $[\rho_{M(t)}]_\alpha$
depending only on $M(t)$. If the solution $\rho(t)$ is also $\alpha$-H\"{o}lder
continuous with uniform in time bound of $[\rho(t)]_\alpha$, then by choosing
$u=\rho(t)-\rho_{M(t)}$ in Lemma~\ref{interpolemm} (with $N=1$) and $0<\mathsf{a}
<\min(\alpha,1-s)/2$,  we get
\[
\|\rho(t)-\rho_{M(t)}\|_2 \leq
C \|(-\Delta)^{-s/2}\big(\rho(t)-\rho_{M(t)}\big)\|_2^{\sigma_1} [\rho(t)-\rho_{M(t)}]_\alpha^{\sigma_2}
\|\rho(t)-\rho_{M(t)}\|_1^{\sigma_3}.
\]
Using~\eqref{frac_lap_rel_entropy}, the uniform in time bound of
$[\rho(t)-\rho_{M(t)}]_\alpha$ and the fact that
$\|\rho(t)\|_{L^1}+\|\rho_{M(t)}\|_{1} \leq 2M_0$, we get
\begin{equation*}\label{estimL2}
\|\rho(t)-\rho_{M(t)}\|_2 \leq C
\Big(H[\rho|\rho_{M(t)}]\Big) ^{\frac{\sigma_1}{2}},
\end{equation*}
for some constant $C$ depending on $\alpha$, $\mathsf{a}$, $M_{0}$. Combining this last inequality with the exponential convergence \eqref{exponentialconv} of the relative entropy $H[\rho(t)|\rho_{M(t)}]$ we obtain the following convergence rate in $L^2$ norm,
\begin{equation}
	\|\rho(t)-\rho_{M(t)}\|_2 \leq C(1+t)^{\sigma_1}
	\exp \Big(-\sigma_1\min{\left(1,\lambda\delta\right)}t\Big).\label{L2 estimate}
	\end{equation}

The convergence in the $L^1(\mathbb{R})$ norm is more involved, because uniform bounds in higher order moments are needed. Arguing as in \cite[Corollary 3.5]{MR3279352} and using \cite[Lemma 2.24]{CarTosc}, we find the following interpolation
between the $L^1(\mathbb{R}^N)$ and $L^2(\mathbb{R}^N)$ norms,
\begin{equation*}
\|\rho(t)-\rho_{M(t)}\|_1\leq C(N,n) \|\rho(t)-\rho_{M(t)}\|_2^{\frac{4n}{4n+1}}\left(\int |x|^{2n}(\rho+\rho_{M(t)})\,\mathrm{d}x\right)^{\frac{1}{1+4n}}.
\end{equation*}
Provided that $u_{0}\in L^1\left(\mathbb{R}, (1+|x|^{2n})\mathrm{d}x\right)\bigcap L^{\infty}(\mathbb{R})$, Lemma \ref{lem:2nm}  guarantees  uniform in time bound of the $2n$-th moment. Therefore, the convergence in $L^2(\mathbb{R})$ norm in \eqref{L2 estimate} yields finally to  the convergence in $L^{1}(\mathbb{R})$ , i.e.
\begin{equation*}
	\|\rho(t)-\rho_{M(t)}\|_1\leq C
	(1+t)^{\frac{4n\sigma_1}{1+4n}}
	\exp \Big(-\frac{4n\sigma_1}{1+4n}\min{\left(1,\lambda\delta\right)}t\Big).
\end{equation*}
Using the similarity variables, the  $L^{1}$ convergence for the solutions of Eq. \eqref{eq:pmea} can also be derived,
\begin{equation*}
\|u(\tau)-U_{M(\tau)}\|_1\leq C	(1+ \log\tau)^{\frac{4n\sigma_1}{1+4n}}	(1+\tau)^{-\frac{4n\sigma_1}{1+4n}\min{\left(\frac{1}{\lambda},\delta\right)}}.
\end{equation*}
This convergence, together with the decay of $L^\infty(\mathbb{R})$
norm of the solution $u(\tau)$ in Proposition~\ref{u:estim}, implies
the convergence in $L^{p}(\mathbb{R})$ for any $p\in(1,\infty)$, although the
rate is unlikely to be optimal.

%%%%%%%%%%%%%%%%%%%%%%%%%%%%%%%%%%%%%%%%%%%%%%%%%%%%%%%%%%%%%%%%%%%%%%%%%%%%%%%%%%%%%%%%%%%%%%%%%%%%%%%%%%%%%%%%%%%%%%%%%%%%%%%%%%%%%%%%%%%%%%%%%%%
%%%%%%%%%%%%%%%%%%%%%%%%%%%%%%%%%%%%%%%%%%%%%%%%%%%%%%%%%%%%%%%%%%%%%%%%%%%%%%%%%%%%%%%%%%%%%%%%%%%%%%%%%%%%%%%%%%%%%%%%%%%%%%%%%%%%%%%%%%%%%%%%%%%

%%%%%%%%%%%%%%%%%%%%%%%%%%%%%%%%%%%%%%%%%%%%%%%%%%%%%%%%%%%%%%%%%%%%%%%%%%%%%%%%%%%%%%%%%%%%%%%%%%%%%%%%%%%%%%%%%%%%%%%%%%%%%%%%%%%%%%%%%%%%%%%%%%%%%%%%%%%%%%%%%%
\section{Fractional diffusion equation with convection}

In this section we focus on the long term asymptotic behaviours of solutions
to the nonlocal porous medium equation~\eqref{eq:pmec} with convection.
As in the previous case with absorption,  basic properties of the solutions
will be reviewed first, followed by the proof of exponential convergence of
relative entropy for smooth solutions and then for more general weak solutions
through limiting process. Detailed proofs in some of the statements below
 will be omitted, if they are straightforward or similar to the case with
 absorption.

Using the same similarity variables~\eqref{eq:simivar}, Eq.~\eqref{eq:pmec} becomes
\begin{equation}\label{eq:smconv}
  \rho_t - \nabla\cdot(\rho \nabla (-\Delta)^{-s}\rho + x\rho )
  = - P(t)^{-\theta}\mathbf{b}\cdot\nabla \rho^q,
\end{equation}
with $P(t)= 1+\lambda\tau=e^{\lambda t}$, $\theta=(N(q-1)+1)/\lambda-1$ and $\rho(x,0)=u_0(x)$.
Here we are interested in the diffusion dominated regime with $\theta>0$, or equivalently
$q>(2N+1-2s)/N$.
The divergence structure of Eq.~\eqref{eq:smconv} implies the conservation of total mass,
namely for all $t \geq0$,
\[
M_0 :=\int u_{0}(x)\mathrm{d}x
= \int \rho(x,t)\mathrm{d}x.
\]
Arguing in the same way as in Proposition \ref{ro:estim}, we get the same estimates.
\begin{prop}\label{ro:estim2}
 Let $\rho$ be a solution to problem \eqref{eq:smconv}. Then estimates \eqref{estim::lp}, \eqref{bound:lp}, \eqref{bound::lp}, \eqref{Linfty_uniform_bound} still hold.
\end{prop}

We will focus on the one dimensional case below with $\mathbf{b}=1$.
It turns out that in this special case, weak solutions
to Eq.~\eqref{eq:pmec} (and hence weak solutions to Eq.~\eqref{eq:smconv}) are unique,
by introducing the  variable in the \emph{integrated form}, i.e.,
\[
  v(x,\tau)=\int_{-\infty}^{x}u(y,\tau)\mathrm{d}y \,\,\hbox{ for }\,\,\tau\geq0,\,x\in\mathbb{R}.
\]
As a result, if $u$ satisfies Eq.~\eqref{eq:pmec}, then  $v$ satisfies the equation
\begin{equation}
v_{\tau}+|v_{x}|(-\Delta)^{1-s}v+|v_{x}|^{q}=0
\label{integratedequat}
\end{equation}
with the initial data $v(x,0)=v_{0}(x)=\int_{-\infty}^{x}u_{0}(y)\mathrm{d}y$.
The mass conservation of $u$ implies the following boundary conditions for $v$,
\[
\lim_{x\rightarrow-\infty}v(x,\tau)=0, \quad \lim_{x\rightarrow+\infty}v(x,\tau)=M_0
\]
for all $\tau\geq0$.  In fact, $v$ is a viscosity solution of Eq.~\eqref{integratedequat}
by following the procedures in~\cite[Proposition 4.2]{SDV18} (see also \cite[Section 8]{MR3419724}).
The uniqueness of these viscosity solutions is then guaranteed by the comparison principle
proved in~\cite[Theorem 6.1]{ChasJako}, yielding to the following result.
\begin{lem}
  Let $u_{0}$ be a nonnegative function in $L^1(\mathbb{R})$. Then there exist
  a unique weak solution to Eq. \eqref{eq:pmec} with initial data $u_0$.
\end{lem}

We now prove the exponential convergence of the relative entropy between the solution $\rho$  and its Barenblatt profile $\rho_{M_0}$.
Formally, if $\rho$ is a smooth solution to Eq.~\eqref{eq:smconv}, then
\begin{align*}
  \frac{\mathrm{d}}{\mathrm{d}t}H[\rho]
=-I[\rho]-P(t)^{-\theta}\int (\rho^{q})_x\left((-\Delta)^{-s}\rho+\frac{1}{2}|x|^2
\right)\mathrm{d}x.
\end{align*}
Now the key step to obtain a self-contained differential inequality  is to
relate the last integral above to $I[\rho]$ (instead of $H[\rho]$ as in the
absorption case) via integration by parts and
 the Cauchy-Schwarz inequality. That is,
\begin{align*}
- P(t)^{-\theta}\int (\rho^q)_x \left[ \frac{x^2}{2}+(-\Delta)^{-s}\rho\right] dx
&=P(t)^{-\theta}\int \rho^q\left[\frac{x^2}{2}+(-\Delta)^{-s}\rho\right]_{x}dx \cr
&\leq P(t)^{-\theta}\left(\int_{\mathbb{R}} |\rho|^{2q-1}\mathrm{d}x\right)^{1/2}I[\rho]^{1/2}\cr
&\leq C(\|u_0\|_1,\|u_0\|_\infty,q,s)P(t)^{-\theta} I[\rho]^{1/2}.
\end{align*}
Since the total mass is conserved, $H[\rho_{M_0}]$ is a constant and the above computation implies
\begin{equation}\label{estim2}%\\
  \frac{\mathrm{d}}{\mathrm{d}t}H[\rho(t)|\rho_{M_0}]
\leq -I[\rho(t)]+CP(t)^{-\theta} I[\rho(t)]^{\frac{1}{2}},
\end{equation}
from which the exponential convergence of the relative entropy can be proved
as in the following theorem.
\begin{thm}
  \label{thm:expsmoothconv}
  Let $N=1$, $q>3-2s$ and $\rho$ be the unique weak solution of the
  one-dimensional nonlocal porous medium equation~\eqref{eq:smconv} with
  non-negative initial data $u_0\in L^{1}(\mathbb{R},(1+|x|^{2})dx)\cap
  L^{\infty}(\mathbb{R})$ and $s<1/2$. Assume that $\rho$ is smooth and let
  $\rho_{M_0}$ be the Barenblatt profile with total conserved mass $M_0$.
  Then the relative entropy $H[\rho|\rho_{M_0}]$ decays to zero exponentially fast. More precisely,
  there is a constant $C$ depending on $\|u_0\|_1, \|u_0\|_{\infty}$,
  $H[u_{0}|\rho_{M_0}]$, $q$ and $s$, such that
  \begin{equation}
    H[\rho(t)|\rho_{M_0}] \leq C(1+t)^2 \exp(-2\min(1,\lambda\theta)t).
    \label{exponentialconv2}
  \end{equation}
\end{thm}

\begin{proof}
Define the function $f(\sigma)=-\sigma+CP(t)^{-\theta} \sigma^{1/2}$
for non-negative $\sigma$. Then $f$ is decreasing on the interval $[
C^{2}P(t)^{-2\theta}/4,\infty)$. If $H[\rho (t)|\rho_{M_0}]  \leq
C^2P(t)^{-2\theta}/8$, then the bound~\eqref{exponentialconv2}
holds true.
Otherwise if $H[\rho(t)|\rho_{M_0}]  > C^2P(t)^{-2\theta}/8$ at any time $t$,
 by the entropy-entropy dissipation inequality \eqref{eq:entrdisp},
\[
C^{2}P(t)^{-2\theta}/4<2 H[\rho(t) |\rho_{M_0}]\leq I[\rho(t)]
\]
and  hence $f\big(I[\rho(t)]\big) \leq f\big(2H[\rho(t)|\rho_{M_0}\big)$
  by the monotonicity  of $f$.
Therefore Eq.~\eqref{estim2} implies
\[
  \frac{\mathrm{d}}{\mathrm{d}t} H[\rho|\rho_{M_0}]  \leq
  f\Big( 2H[\rho|\rho_{M_0}]\Big) = -2H[\rho|\rho_{M_0}]
  +\sqrt{2}CP(t)^{-\theta}\big(H[\rho|\rho_{M_0}]\big)^{1/2}.
\]
This differential inequality can be solved in the same way as in
Theorem~\ref{expdecaysmooth} to obtain the
exponential bound~\eqref{exponentialconv2}
of the relative entropy.
\end{proof}

To show the exponential convergence of general weak solutions of
Eq.~\eqref{eq:smconv}, similarly we consider the one-dimensional regularised problem
\begin{equation}\label{eq:smregconv}
\frac{\partial \rho^\epsilon}{\partial t}
=\nabla\cdot\left[ \rho^\epsilon \nabla\left( (-\Delta)^{-s}\rho^\epsilon
+\frac{1}{2}|x|^2 + \epsilon \log \rho^\epsilon\right)\right] -
P(t)^{-\theta}\frac{\partial}{\partial x}(\rho^\epsilon)^q,\qquad \rho^\epsilon(x,0) = u_0(x).
\end{equation}
Many properties of the solution $\rho^\epsilon$ can be established in a similar way
in the absorption case, such as the counterpart of Lemma \ref{boundsecmomapr},
the confinement of total mass and the weak $L^1$ convergence of
$\rho^\epsilon(t)$ to the weak solution $\rho(t)$ as $\epsilon$ goes to zero.
Moreover, from the governing equation~\eqref{eq:smregconv} satisfied by
$\rho^\epsilon$, we get
\[
  \frac{\mathrm{d}}{\mathrm{d}t}
  H_\epsilon\big[\rho^\epsilon(t)|\rho^\epsilon_{M_0}\big]
  \leq -I_\epsilon[\rho^\epsilon(t)]
  +CP(t)^{-\theta}I_\epsilon[\rho^\epsilon(t)]^{1/2}.
\]
Using the same technique as in Theorem~\ref{thm:expsmoothconv},
we obtain the exponential convergence of
$H_{\epsilon}[\rho^{\epsilon}(t)|\rho^\epsilon_{M_0}]$, and hence
the same bound~\eqref{exponentialconv2}  for general weak solutions by taking the limit as
$\epsilon$ goes to zero, providing the proof of Theorem~\ref{thm:conv} for the convergence of relative entropy for the original equation~\eqref{eq:pmec}.
Finally, by assuming uniform in time H\"{o}lder seminorm of
the solution $\rho(t)$, the exponential convergence of $\rho(t)$
towards $\rho_{M_0}$ in $L^1(\mathbb{R})$ or $L^2(\mathbb{R})$ norms  can be proved in a similar as in Section~\ref{Lpdecay}.

%in order to derive the convergence rate of the solutions $u$ to \eqref{eq:pmec}, we
%can just rescale back in the self similarity variables and argue as in Section 3.4,
%replacing $\lambda$ by $\theta$, provided $u$ is $C^{\alpha}$ for some $\alpha\in (0,1)$. Then assuming $u_{0}\in L^1\left(\mathbb{R}, (1+|x|^{2n})dx\right)\cap L^{\infty}(\mathbb{R})$ we have, for some constant $C>0$ and positive exponents $\gamma_1$, $\gamma_2$
%\begin{equation*}
%\begin{split}
%\|u(t)-U_{M_0}\|_1\leq C
%(1+ \log\tau)^{\gamma_1}
%(1+\tau)^{-\gamma_2},
%\end{split}
%\end{equation*}
%where $U_{M_0}$ denotes the Barenblatt profile of mass $M_0$ of \eqref{eq:pmep}. Interpolating between $L^{1}$ and $L^{\infty}$ it is easy to obtain the intermediate $L^{p}$ asymptotics for all $p\in (1,\infty)$.

%%%%%%%%%%%%%%%%%%%%%%%%%%%%%%%%%%%%%%%%%%%%%%%%%%%%%%%%%%%%%%%%%%%%%%%%%%%%%%%%%%%%%%%%%%%%%%%%
\section{Conclusion and generalisation}\label{generalisation}

In this paper, the long time behaviours of the solutions to the nonlocal porous medium equation
are studied, by showing the convergence of the relative entropy between the solutions and their
Barenblatt profiles. The convergence in other norms can also be obtained, by assuming additional H\"{o}lder regularity on the solutions.
Although we only concentrated on equations with power-law type absorption or convection,
the same procedures can be applied to a larger class of models. For the general one-dimensional equation
\[
  \rho_t - \nabla\cdot (\rho \nabla (-\Delta)^{-s} \rho + x\rho) = - P(t)^{-\delta}
  g(t,x,\rho)
\]
with absorption $g(t,x,\rho)\geq 0$, formally we  have
\begin{align*}
  \frac{\mathrm{d}}{\mathrm{d}t} H[\rho(t)|\rho_{M(t)}]
  &\leq -I[\rho(t)] - P(t)^{-\delta}
  \int (-\Delta)^{-s}(\rho(t)-\rho_{M(t)})g(t,x,\rho) \cr
  &\leq -2H[\rho(t)|\rho_{M(t)}] + P(t)^{-\delta}\big\|(-\Delta)^{-s/2}\big(\rho(t)-\rho_{M(t)}\big)\big\|_2
  \big\|(-\Delta)^{-s/2}g(t,\cdot,\rho)\big\|_2.
\end{align*}
Similarly, for the equation
\[
  \rho_t - \nabla\cdot (\rho \nabla (-\Delta)^{-s} \rho +
  x\rho) =-P(t)^{-\theta}\frac{\partial}{\partial x}\big(\rho^{1/2}h(t,x,\rho)\big),
\]
with general convection $\rho^{1/2}h(t,x,\rho)$ we obtain
\begin{align*}
  \frac{\mathrm{d}}{\mathrm{d}t} H[\rho(t)|\rho_{M_0}] \leq  -I[\rho(t)] + P(t)^{-\theta}
  \big\|h(t,\cdot,\rho)\big\|_2I[\rho(t)]^{1/2}.
\end{align*}
As long as the solution $\rho$, the absorption $g$ or the convection $\rho^{1/2}h$
satisfy appropriate bounds,  the exponential
convergence of the relative entropy can also be established.

In addition to the above generalisations with other absorption or convection terms,
many related problems are still widely open.
The convergence results presented here can not be extended into higher dimensions,
precisely because the entropy-entropy dissipation inequality is only proved in one dimension.
Therefore, the clarification of this critical inequality in higher dimensions will shed lights on
the behaviours of solutions to other associated equations.
In the parameter regimes we considered in this paper, the absorption or convection eventually
becomes  exponentially small in the transformed equation with the similarity variables.
In the other parameter regimes, the strength
of the absorption or the convection becomes comparable to the nonlocal diffusion,
usually leading to new nontrivial equations whose quantitative properties are much
more difficult to study.

%%%%%%%%%%%%%%%%%%%%%%%%%%%%%%%%%%%%%%%%%%%%%%%%%%%%%%%%%%%%%%%%%%%%%%%%%%%%

\paragraph*{Acknowledgements}
We thank Prof. Juan Luis V\'{a}zquez for fruitful discussions concerning absorption-diffusion and convection-diffusion models along with some regularity problems connected to model~\eqref{eq:pmep}. We thank Felix Del Teso for useful discussions about the one dimensional integrated model obtained from \eqref{eq:pmep} and related interesting numerical aspects.
This work has been partially supported by GNAMPA of the Italian
INdAM (National Institute of High Mathematics) and ``Programma
triennale della Ricerca dell'Universit\`{a} degli Studi di Napoli
``Parthenope'' - Sostegno alla ricerca individuale 2015-2017''.

%\bibliographystyle{plain}
%\bibliography{biofrac}

\end{document}